\documentclass[11pt]{amsart}

\usepackage[
paper=a4paper,
text={147mm,230mm},centering
]{geometry}

\usepackage{amssymb}
\usepackage[all]{xy}
\usepackage{graphicx,color,float}
\usepackage[bookmarks]{hyperref}
\usepackage{pinlabel}
\usepackage[textwidth=1.3in,color=yellow]{todonotes}

\iftrue
\makeatletter
\def\@settitle{%
  \vspace*{-20pt}
  \begin{flushleft}%
    \baselineskip14\p@\relax
    \normalfont\bfseries\LARGE
    \@title
  \end{flushleft}%
}
\def\@setauthors{%
  \begingroup
  \def\thanks{\protect\thanks@warning}%
  \trivlist
  \large \@topsep30\p@\relax
  \advance\@topsep by -\baselineskip
  \item\relax
  \author@andify\authors
  \def\\{\protect\linebreak}%
  \authors
  \ifx\@empty\contribs
  \else
    ,\penalty-3 \space \@setcontribs
    \@closetoccontribs
  \fi
  \normalfont
  \@setaddresses
  \endtrivlist
  \endgroup
}
\def\@setaddresses{\par
  \nobreak \begingroup\raggedright
  \small
  \def\author##1{\nobreak\addvspace\smallskipamount}%
  \def\\{\unskip, \ignorespaces}%
  \interlinepenalty\@M
  \def\address##1##2{\begingroup
    \par\addvspace\bigskipamount\noindent
    \@ifnotempty{##1}{(\ignorespaces##1\unskip) }%
    {\ignorespaces##2}\par\endgroup}%
  \def\curraddr##1##2{\begingroup
    \@ifnotempty{##2}{\nobreak\noindent\curraddrname
      \@ifnotempty{##1}{, \ignorespaces##1\unskip}\/:\space
      ##2\par}\endgroup}%
  \def\email##1##2{\begingroup
    \@ifnotempty{##2}{\smallskip\nobreak\noindent E-mail address%
      \@ifnotempty{##1}{, \ignorespaces##1\unskip}\/:\space
      \ttfamily##2\par}\endgroup}%
  \def\urladdr##1##2{\begingroup
    \def~{\char`\~}%
    \@ifnotempty{##2}{\nobreak\noindent\urladdrname
      \@ifnotempty{##1}{, \ignorespaces##1\unskip}\/:\space
      \ttfamily##2\par}\endgroup}%
  \addresses
  \endgroup
  \global\let\addresses=\@empty
}
\def\@setabstracta{%
    \ifvoid\abstractbox
  \else
    \skip@25\p@ \advance\skip@-\lastskip
    \advance\skip@-\baselineskip \vskip\skip@
    \box\abstractbox
    \prevdepth\z@ % because \abstractbox is a vtop
    \vskip-10pt
  \fi
}
\renewenvironment{abstract}{%
  \ifx\maketitle\relax
    \ClassWarning{\@classname}{Abstract should precede
      \protect\maketitle\space in AMS document classes; reported}%
  \fi
  \global\setbox\abstractbox=\vtop \bgroup
    \normalfont\small
    \list{}{\labelwidth\z@
      \leftmargin0pc \rightmargin\leftmargin
      \listparindent\normalparindent \itemindent\z@
      \parsep\z@ \@plus\p@
      
    }%
    \item[\hskip\labelsep\bfseries\abstractname.]%
}{%
  \endlist\egroup
  \ifx\@setabstract\relax \@setabstracta \fi
}
\def\section{\@startsection{section}{1}%
  \z@{-1.2\linespacing\@plus-.5\linespacing}{.8\linespacing}%
  {\normalfont\bfseries\large}}
\def\subsection{\@startsection{subsection}{2}%
  \z@{-.8\linespacing\@plus-.3\linespacing}{.3\linespacing\@plus.2\linespacing}%
  {\normalfont\bfseries}}
\def\subsubsection{\@startsection{subsubsection}{3}%
  \z@{.7\linespacing\@plus.1\linespacing}{-1.5ex}%
  {\normalfont\itshape}}
\def\@secnumfont{\bfseries}
\makeatother
\fi %\iftrue/false for amsart.cls hack

\def\to{\mathchoice{\longrightarrow}{\rightarrow}{\rightarrow}{\rightarrow}}
\makeatletter
\newcommand{\shortxra}[2][]{\ext@arrow 0359\rightarrowfill@{#1}{#2}}
\def\longrightarrowfill@{\arrowfill@\relbar\relbar\longrightarrow}
\newcommand{\longxra}[2][]{\ext@arrow 0359\longrightarrowfill@{#1}{#2}}
\renewcommand{\xrightarrow}[2][]{\mathchoice{\longxra[#1]{#2}}%
  {\shortxra[#1]{#2}}{\shortxra[#1]{#2}}{\shortxra[#1]{#2}}}
\makeatother

\def\otimesover#1{\mathbin{\mathop{\otimes}_{#1}}}

\makeatletter
\def\Nopagebreak{\@nobreaktrue\nopagebreak}
\makeatother

\theoremstyle{plain}
\newtheorem{theorem}{Theorem}[section]
\newtheorem{proposition}[theorem]{Proposition}
\newtheorem{corollary}[theorem]{Corollary}
\newtheorem{lemma}[theorem]{Lemma}

\theoremstyle{definition}
\newtheorem{definition}[theorem]{Definition}

\newtheorem{example}[theorem]{Example}
\newtheorem{remark}[theorem]{Remark}

\def\Z{\mathbb{Z}}
\def\Q{\mathbb{Q}}
\def\R{\mathbb{R}}
\def\C{\mathbb{C}}

\def\N{\mathcal{N}}

\def\K{\mathcal{K}}

\def\cC{\mathcal{C}}
\def\cR{\mathcal{R}}
\def\cP{\mathcal{P}}

\def\I{\mathcal{I}}

\def\Ker{\operatorname{Ker}}
\def\Coker{\operatorname{Coker}}
\def\Im{\operatorname{Im}}
\def\Hom{\operatorname{Hom}}
\def\Tor{\operatorname{Tor}}

\def\sign{\operatorname{sign}}
\def\rank{\operatorname{rank}}

\def\inte{\operatorname{int}}

\def\Bl{B\ell}
\def\lk{\operatorname{lk}}

\def\rhot{\rho^{(2)}}
\def\dimt{\dim^{(2)}}
\let\ldim=\dimt
\def\bt{b^{(2)}}

\def\lsign{\sign^{(2)}}

\def\mathbinover#1#2{\mathbin{\mathop{#1}\limits_{#2}}}
\def\amalgover#1{\mathbinover{\amalg}{#1}}
\def\cupover#1{\mathbinover{\cup}{#1}}

\begin{document}

\title
[Symmetric Whitney tower cobordism] 
{Symmetric Whitney tower cobordism for bordered 3-manifolds and links}

\author{Jae Choon Cha}

\address{Department of Mathematics\\
  POSTECH \\
  Pohang 790--784\\
  Republic of Korea\\
  and\linebreak
  School of Mathematics\\
  Korea Institute for Advanced Study \\
  Seoul 130--722\\
  Republic of Korea}
\email{jccha@postech.ac.kr}

\def\subjclassname{\textup{2010} Mathematics Subject Classification}
\expandafter\let\csname subjclassname@1991\endcsname=\subjclassname
\expandafter\let\csname subjclassname@2000\endcsname=\subjclassname
\subjclass{%
  57M25, % Knots and links in $S^3$
  57M27, % Invariants of knots and 3-manifolds
  57N70. % Cobordism and concordance (in low dimension)
%  57Q60; % Cobordism and concordance (in high dimension)
%  57M07, % Topological methods in group theory
}

\keywords{Whitney towers, Gropes, Link concordance, Homology
  cobordism, Amenable $L^2$-signatures}

\begin{abstract}
  We introduce a notion of symmetric Whitney tower cobordism between
  bordered 3-manifolds, aiming at the study of homology cobordism and
  link concordance.  It is motivated by the symmetric Whitney tower
  approach to slicing knots and links initiated by T. Cochran, K. Orr,
  and P. Teichner.  We give amenable Cheeger-Gromov $\rho$-invariant
  obstructions to bordered 3-manifolds being Whitney tower cobordant.
  Our obstruction is related to and generalizes several prior known
  results, and also gives new interesting cases.  As an application,
  our method applied to link exteriors reveals new structures on
  (Whitney tower and grope) concordance between links with nonzero
  linking number, including the Hopf link.
\end{abstract}

% \noindent\emph{Not for general distribution}
% \vspace{5ex}

\maketitle

\section{Introduction}

It is well known that \emph{Whitney towers} and \emph{gropes} play a
key r\^ole in several important problems in low dimensional topology,
particularly in the study of topology of 4-manifolds and concordance
of knots and links.  Whitney towers and gropes approximate embedded
2-disks, 2-spheres, and more generally embedded surfaces, in a
4-manifold.  Roughly, a Whitney tower can be viewed as (the trace of)
an attempt to apply Whitney moves repeatedly to remove intersection
points of immersed surfaces in dimension~4; it consists of various
layers of immersed Whitney disks which pair up intersection points of
prior layers.  A grope in a 4-manifold consists of embedded surfaces
with disjoint interiors which represent essential curves on prior
layer surfaces as commutators.

In this article we are interested in \emph{symmetric} Whitney towers
and gropes, which have a \emph{height}.  These are analogous to the
commutator construction of the \emph{derived series}.  We remark that
Whitney towers and gropes related to the \emph{lower central series}
are also often considered.  Although these Whitney towers and gropes
still give interesting structures concerning links and 4-dimensional
topology (for example, see the recent remarkable work of J. Conant,
R. Schneiderman, and P. Teichner surveyed
in~\cite{Conant-Teichner-Schneiderman:2011-1}), it is known that
symmetric Whitney towers and gropes are much closer approximations to
embedded surfaces that give extremely rich theory.

Our main aim is to study homology cobordism of 3-manifolds with
boundary using symmetric Whitney towers in dimension~4 and amenable
Cheeger-Gromov $\rho$-invariants.  Our setting is strongly motivated
by the symmetric Whitney tower approach to the knot (and link) slicing
problem which was first initiated by T. Cochran, K. Orr, and
P. Teichner~\cite{Cochran-Orr-Teichner:1999-1}, and the amenable
$L^2$-theoretic technique for the Cheeger-Gromov $\rho$-invariants due
to Orr and the author~\cite{Cha-Orr:2009-01}.  As a new application
that known Whitney tower frameworks do not cover, we study concordance
between links with nonzero linking number.  In particular we
investigate Whitney tower and grope concordance to the Hopf link.

\subsubsection*{Symmetric Whitney tower cobordism of bordered 3-manifolds}

First we introduce briefly how we adapt the Whitney tower approach in
\cite{Cochran-Orr-Teichner:1999-1} for homology cobordism of bordered
3-manifolds.  

Recall that a 3-manifold $M$ is \emph{bordered} by a surface $\Sigma$
if it is endowed with a marking homeomorphism of $\Sigma$ onto
$\partial M$.  For 3-manifolds $M$ and $M'$ bordered by the same
surface, one obtains a closed 3-manifold $M\cup_\partial -M'$ by
glueing the boundary along the marking homeomorphism.  A 4-manifold
$W$ is a \emph{relative cobordism} from $M$ to $M'$ if $\partial
W=M\cup_\partial -M'$.

A relative cobordism $W$ from $M$ to $M'$ is a \emph{homology
  cobordism} if the inclusions induce isomorphisms $H_*(M;\Z)\cong
H_*(W;\Z)\cong H_*(M';\Z)$.  Initiated by S. Cappell and
J. Shaneson~\cite{Cappell-Shaneson:1974-1}, understanding homology
cobordism of bordered manifolds is essential in the study of manifold
embeddings, in particular knot and link concordance.  This also
relates homology cobordism to other key problems including topological
surgery on 4-manifolds.
 
As a surgery theoretic Whitney tower approximation to a homology
cobordism, we will define the notion of a \emph{height $h$ Whitney tower
  cobordism} $W$ between bordered 3-manifolds ($h\in \frac12 \Z_{\ge
  0}$).  Roughly speaking, our height $h$ Whitney tower cobordism is a
relative cobordism between bordered 3-manifolds, which admits immersed
framed 2-spheres satisfying the following: while the 2-spheres may not
be embedded, these support a Whitney tower of height~$h$, and form a
``lagrangian'' in such a way that if the 2-spheres were homotopic to
embeddings then surgery along these would give a homology cobordism.
For a more precise description of a Whitney tower cobordism, see
Definition~\ref{definition:Whitney-tower-cobordism}.

It turns out that a height $h$ Whitney tower cobordism can be deformed
to another type of a relative cobordism satisfying a twisted homology
analogue of the above Whitney tower condition, which we call an
\emph{$h$-solvable cobordism}.  Roughly speaking, it is a cobordism
which induces an isomorphism on $H_1$ and admits a certain
``lagrangian'' and ``dual'' for the twisted intersection pairing
associated to (quotients by) derived subgroups of the fundamental
group.  See
Definition~\ref{definition:lagrangian-dual-solvable-cobordism} and
Theorem~\ref{theorem:whitney-tower-and-solvability} for details.  Our
$h$-solvable cobordism can be viewed as a relative version of the
notion of an $h$-solution first introduced
in~\cite{Cochran-Orr-Teichner:2002-1}.

\subsubsection*{Amenable signature theorem}

In order to detect the non-existence of a Whitney tower cobordism, we
show that certain \emph{amenable $L^2$-signatures}, or equivalently
\emph{Cheeger-Gromov $\rhot$-invariants}, give obstructions to
bordered 3-manifolds being height $h$ Whitney tower cobordant.
Interestingly, for the height $n.5$ obstructions stated below, we have
two alternative hypotheses on the first $L^2$-Betti numbers:
\emph{zero} or \emph{large enough}.  In what follows $\bt_i(-;\N G)$
and $b_i(-;R)$ denote the $L^2$-Betti number over $\N G$ and the Betti
number over~$R$, namely the $L^2$- and $R$-dimension of $H_i(-;\N G)$
and $H_i(-;R)$ respectively.

\newtheorem*{amenable-signature-theorem}
{Theorem~\ref{theorem:amenable-signature-for-solvable-cobordism}}

\begin{amenable-signature-theorem}[Amenable Signature Theorem for
  solvable cobordism]
  Suppose $W$ is a relative cobordism between two bordered 3-manifolds
  $M$ and~$M'$, $G$ is an amenable group lying in Strebel's class
  $D(R)$, $R=\Z/p$ or $\Q$, and $G^{(n+1)}=\{e\}$.  Suppose
  $\phi\colon \pi_1(M\cup_\partial -M') \to G$ factors through
  $\pi_1(W)$, and either one of the following conditions holds:
\begin{enumerate}
\item[(I)] $W$ is an $n.5$-solvable cobordism and $\bt_1(M;\N G)=0$;
\item[(II)] $W$ is an $n.5$-solvable cobordism,
  $|\phi(\pi_1(M))|=\infty$, and
  \[
  \bt_1(M\cup_\partial -M';\N G) \ge b_1(M;R)+b_2(M;R)+b_3(M;R)-1
  \text{; or}
  \]
\item[(III)] $W$ is an $(n+1)$-solvable cobordism.
\end{enumerate}
Then the Cheeger-Gromov invariant $\rhot(M\cup_\partial -M',\phi)$
vanishes.
\end{amenable-signature-theorem}

For the definition of amenable
groups and Strebel's class $D(R)$~\cite{Strebel:1974-1}, see
Definition~\ref{definition:amenability-strebel-D(R)}.  To prove
Amenable Signature
Theorem~\ref{theorem:amenable-signature-for-solvable-cobordism}, we
use extensively the $L^2$-theoretic techniques developed by Orr and
the author in~\cite{Cha-Orr:2009-01},~\cite{Cha:2010-01}.  For more
details and related discussions, see
Section~\ref{section:amenable-signature-theorem}.

Amenable Signature
Theorem~\ref{theorem:amenable-signature-for-solvable-cobordism}
generalizes several prior known cases (discussed in more detail in
Section~\ref{subsection:relationship-with-previous-results}).  First,
it specializes to the amenable signature obstructions to knots being
$n.5$-solvable given in \cite{Cha:2010-01}, and Cochran-Orr-Teichner's
PTFA signature obstructions \cite{Cochran-Orr-Teichner:1999-1}.  Also,
from our result it follows that Harvey's homology cobordism invariant
$\rho_n$ for closed 3-manifolds \cite{Harvey:2006-1} associated to her
torsion-free derived series is an obstruction to being Whitney tower
cobordant.

Moreover, Amenable Signature
Theorem~\ref{theorem:amenable-signature-for-solvable-cobordism} for
the condition (I) provides an interesting new case.  In
Section~\ref{subsection:vanishing-of-b^2_1} we discuss some instances
of bordered 3-manifolds for which the first $L^2$-Betti number
vanishes.  This will be used to give applications to links with
nonvanishing linking number, as described below.

\subsubsection*{Symmetric Whitney tower concordance of links}
 
Our setting for bordered 3-manifolds is useful in studying geometric
equivalence relations of links defined in terms of Whitney towers and
gropes.  We recall that two $m$-component links $L$ and $L'$ in $S^3$
are \emph{concordant} if there are $m$ disjointly embedded locally
flat annuli in $S^3\times[0,1]$ cobounded by components of $L\times0$
and $-L'\times 1$.  Again, approximating embedded annuli by Whitney
towers, one defines \emph{height $h$ (symmetric) Whitney tower
  concordance}: embedded annuli in the definition of concordance are
replaced with transverse immersed annuli which admit a Whitney tower
of height~$h$ (see
Definition~\ref{definition:whitney-tower-concordance}).  \emph{Height
  $h$ (symmetric) grope concordance} between links is defined
similarly, replacing disjoint annuli with disjoint height $h$ gropes
(see Definition~\ref{definition:grope-concordance}).

Schneiderman showed that if $L$ and $L'$ are height $h$ grope
concordant, then these are height $h$ Whitney tower
concordant~\cite{Schneiderman:2006-1}.  Furthermore, following the
lines of \cite{Cochran-Orr-Teichner:1999-1}, one can observe that if
two links are height $h+2$ Whitney tower concordant, then their
\emph{exteriors} are, as bordered 3-manifolds, height $h$ Whitney
tower cobordant (see
Theorem~\ref{theorem:link-whitney-tower-concobordism}).  Therefore
Amenable Signature
Theorem~\ref{theorem:amenable-signature-for-solvable-cobordism} gives
obstructions to links being Whitney tower (and grope) concordant.

Summarizing, we have the implications illustrated in
Figure~\ref{figure:whitney-tower-grope-amenable-signature}.

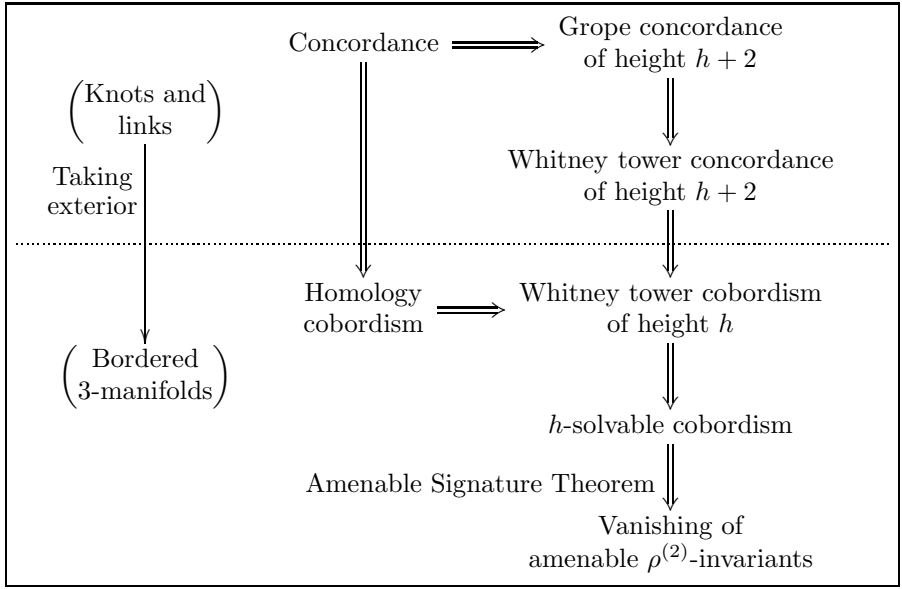
\begin{figure}[htb]
\fbox{\small$
\xymatrix@R=0em@C=1.5em@M=.05em{
&
&
{\begin{tabular}{c}Concordance\end{tabular}}
\ar@{=>}[r]
\ar@{=>}[dddd]
&
{\begin{tabular}{c}Grope concordance\\of height $h+2$\end{tabular}}
\ar@{=>}[dd]
\\
&
{\left(\begin{tabular}{@{}c@{}}Knots and \\links\end{tabular}\right)}
\ar[dddd]_(.3){\txt{Taking\\exterior}}
\\
&
&
&
{\begin{tabular}{c}Whitney tower concordance\\of height $h+2$\end{tabular}}
\ar@{=>}[dd]
\\
{\vphantom{\begin{array}{c}X\\X\end{array}}}\ar@{.}[rrrr]
&&&&{}
\\
&
&
{\begin{tabular}{c}Homology\\cobordism\end{tabular}}
\ar@{=>}[r]
&
{\begin{tabular}{c}Whitney tower cobordism\\of height $h$\end{tabular}}
\ar@{=>}[dd]
\\
&
{\left(\begin{tabular}{@{}c@{}}Bordered\\3-manifolds\end{tabular}\right)}
\\
&
&
&
{\begin{tabular}{c}$h$-solvable cobordism\end{tabular}}
\ar@{=>}[dd]_{\txt{Amenable Signature Theorem}\,\,}
\\
{\vphantom{\begin{array}{c}X\\X\end{array}}}
\\
&
&
&
{\begin{tabular}{c}Vanishing of\\amenable
  $\rhot$-invariants\end{tabular}}
}
$}
\caption{Whitney towers, gropes, and amenable signatures}
\label{figure:whitney-tower-grope-amenable-signature}
\end{figure}

\subsubsection*{Application to links with nonvanishing linking number}

As an application, we investigate concordance of links with
nonvanishing linking number, particularly concordance to the Hopf
link.

There are several known techniques to detect non-concordant links in
the literature, which reveal interesting structures peculiar to links
up to concordance (even modulo knots).  These include classical
abelian invariants such as Fox-Milnor type conditions for the
(reduced) multi-variable Alexander polynomial and Levine-Tristram type
signature invariants (e.g., see K. Murasugi~\cite{Murasugi:1965-1},
A. Tristram~\cite{Tristram:1969-1},
A. Kawauchi~\cite{Kawauchi:1978-1}), and the Witt class of the
Blanchfield pairing (e.g., see J. Hillman's
book~\cite{Hillman:2002-1}).  Also there are various signatures and
twisted torsion invariants associated to various nonabelian ``higher
order'' covers (e.g., see J. Levine
\cite{Levine:1994-1,Levine:2007-1}, Cha-Ko \cite{Cha-Ko:1999-1},
S. Friedl \cite{Friedl:2003-1}, S. Harvey \cite{Harvey:2006-1},
Cochran-Harvey-Leidy \cite{Cochran-Harvey-Leidy:2008-1},
Cha-Friedl~\cite{Cha-Friedl:2010-01}), and there are Witt-class valued
Hirzebruch-type invariants from iterated $p$-covers (see
\cite{Cha:2007-1,Cha:2007-2}).  These newer techniques are mainly
focused on link slicing problems.  Geometric techniques such as
covering link calculus and blowing down are also known to be useful
in several cases (e.g., see Cochran-Orr~\cite{Cochran-Orr:1993-1},
Cha-Ko~\cite{Cha-Ko:1999-2},
Cha-Livingston-Ruberman~\cite{Cha-Livingston-Ruberman:2006-1},
Cha-Kim~\cite{Cha-Kim:2008-1}).

Compared to the slicing problem, the more general case of concordance
between links with possibly nonvanishing linking number has been less
studied.  Note that the question of whether two given links are
concordant is not directly translated to a link slicing problem, while
for knots it can be done via connected sum.

The case of linking number one two component links has received some
recent attention.  It is remarkable that such links bear a certain
resemblance to knots, where the Hopf link is analogous to the unknot.
In this sense these may be viewed as ``small'' links.  For example,
while in many cases (e.g., for boundary links) there are large
nilpotent representations of the fundamental group of a link
complement which eventually lead us to interesting link concordance
invariants, it can be seen that there is no useful nonabelian
nilpotent representations for two component links with linking number
one, similarly to the knot case.

Classical abelian invariants such as the Alexander polynomial give
primary information for the linking number one case, and as a partial
converse J. Davis showed that any two-component link with Alexander
polynomial one is topologically concordant to the Hopf
link~\cite{Davis:2006-1}.  This result, which may be viewed as a link
analogue of the well-known result of Freedman on Alexander polynomial
one knots~\cite{Freedman:1984-1}, illustrates another similarity
between knots and ``small'' links.  We remark that recently T. Kim,
D. Ruberman, S. Strle, and the author have shown that Davis's result
does not hold in the smooth category, even for links with unknotted
components~\cite{Cha-Kim-Ruberman-Strle:2010-01}.

S. Friedl and M. Powell have recently developed Casson-Gordon style
metabelian link invariants by generalizing techniques for knots, and
detected interesting examples of links not concordant to the Hopf
link~\cite{Friedl-Powell:2010-1,Friedl-Powell:2011-1}.  They
conjectured that their invariant vanishes for links which are height
$3.5$ Whitney tower concordant to the Hopf link.  Recently M. Kim has
confirmed this conjecture, using our framework of Whitney tower
cobordism.  He has also related known abelian invariants to low height
Whitney tower concordance for linking number one links.

Our method reveals new sophisticated structures concerning linking
number one links not concordant to the Hopf link.  This may be viewed
as an analogue of the study of knots using higher order
$L^2$-invariants, beyond abelian and metabelian invariants.  An
advantage of our setup for the higher order invariants is that we can
use the \emph{exterior} of a link instead of the zero-surgery manifold
which is used in recent works on $L^2$-invariants for links.  We
remark that for two component links with linking number one, the
zero-surgery manifold is a homology 3-sphere and consequently has no
interesting solvable representations; this is a reason that
several recent techniques of higher order $L^2$-invariants do not
apply directly to this case.

Using our Amenable Signature
Theorem~\ref{theorem:amenable-signature-for-solvable-cobordism}
applied to link exteriors, we prove the following result:

\newtheorem*{hopf-link-theorem}
{Theorem~\ref{theorem:whitney-tower-concordance-to-hopf-link}}
\begin{hopf-link-theorem}
  For any integer $n>2$, there are links with two unknotted components
  which are height $n$ grope concordant (and consequently height $n$
  Whitney tower concordant) to the Hopf link, but not height $n.5$
  Whitney tower concordant (and consequently not height $n.5$ grope
  concordant) to the Hopf link.
\end{hopf-link-theorem}

We remark that more applications of our methods, including those on
homology cylinders, will be presented in other papers.

Results in this article hold in both topological category (with
locally flat submanifolds) and smooth category.  For a related
discussion, see Remark~\ref{remark:topological-vs-smooth}.  Manifolds
are assumed to be connected, compact, and oriented, and $H_*(-)$
denotes homology with integral coefficients unless stated otherwise.

\subsubsection*{Acknowledgements}
\label{subsubsection:acknowledgements}

The exposition of the paper was significantly improved by detailed
comments of an anonymous referee.  This research was supported in part
by NRF grants 2011--0030044, 2010--0029638, 2010--0011629 funded by
the government of Korea.

\section{Whitney tower cobordism}
\label{section:whitney-tower-and-concordance}

In this section we formulate a notion of symmetric Whitney tower
approximations of homology cobordism.  This gives a setup generalizing
the approach to the knot and link slicing problem initiated
in~\cite{Cochran-Orr-Teichner:1999-1}.  In what follows, to make the
exposition more readable, we discuss some motivations and backgrounds
as well.  Readers familiar with the notion of Whitney towers, gropes,
and the approach of~\cite{Cochran-Orr-Teichner:1999-1} may proceed to
the next section after reading only key definitions and statements of
our setup:
Definitions \ref{definition:0-lagrangian}
(0-lagrangian), \ref{definition:Whitney-tower-cobordism} (Whitney tower cobordism),
\ref{definition:lagrangian-dual-solvable-cobordism} ($h$-solvable
cobordism), Theorems \ref{theorem:whitney-tower-and-solvability}
(Whitney tower cobordism $\Rightarrow$ solvable cobordism),
\ref{theorem:link-whitney-tower-concobordism} (Whitney tower/grope
concordance $\Rightarrow$ Whitney tower cobordism).

\subsection{Homology cobordism and $H_1$-cobordism of bordered
  3-manifolds}

Recall that a relative cobordism $W$ between bordered 3-manifolds $M$
to $M'$ is a manifold with $\partial W=M\cup_\partial -M'$, and that
$W$ is a \emph{(relative) homology cobordism} if $H_*(M)\cong H_*(W)
\cong H_*(M')$ under the inclusion-induced maps.  As an abuse of
notation we often write $M\cup_\partial M'$ instead of $M\cup_\partial
-M'$.  Our primary example of a homology cobordism is obtained from
knots and links.

\begin{example}[Link exterior]
  If $L$ is a link in $S^3$, then the exterior $E_L=S^3-$ (open
  tubular neighborhood of $L$) is a 3-manifold bordered by the
  disjoint union of tori, where the marking is given canonically by
  the 0-linking framing of each component.

  If $L$ is concordant to $L'$, then the concordance exterior (with
  rounded corners) is a relative homology cobordism from $E_L$ to
  $E_{L'}$.  In fact this conclusion holds if $L$ is concordant to
  $L'$ in a homology $S^3\times[0,1]$.
\end{example}

The following is well-known and easily verified: two links $L$ and
$L'$ in $S^3$ are topologically concordant if and only if there is a
topological homology cobordism $W$ from $E_L$ to $E_{L'}$ with
$\pi_1(W)$ normally generated by meridians of~$L$.  The smooth
analogue holds modulo smooth Poincar\'e conjecture in dimension~4.
Also, two links $L$ and $L'$ in homology 3-spheres are
(topologically/smoothly) concordant in a (topological/smooth) homology
$S^3\times[0,1]$ if and only if their exteriors $E_L$ and $E_{L'}$ are
(topologically/smoothly) homology cobordant.

\subsubsection*{Relative $H_1$-cobordism of bordered 3-manifolds }

% Suppose $M$ and $M'$ are 3-manifolds bordered by the same surface
% $\Sigma$, $(X,\Sigma)$ is a 3-dimensional Poincar\'e pair, and
% $(M,\partial M) \to (X,\Sigma)$ and $(M',\partial M') \to (X,\Sigma)$
% are homology equivalences extending the marking maps.  One tries to
% construct a relative homology cobordism between $M$ and~$M'$.
% % It is not difficult to check whether there is a (relative) bordism
% % between $M$ and $M'$ over~$X$.  Furthermore, when there is a
% % relative bordism,
% As an obvious necessary condition suppose there is a (relative)
% bordism between $M$ and $M'$ over $X$.  Then
% %
% one can always do
% surgery-below-middle-dimension (along circles) to obtain a new
% 4-manifold $W$ with desired $H_1$ as described below:

Suppose $M$ and $M'$ are 3-manifolds bordered by the same
surface~$\Sigma$.  As the first step toward a homology cobordism, one
considers the following, generalizing
\cite[Definition~8.1]{Cochran-Orr-Teichner:1999-1}.  (See also
\cite[Definition~2.1]{Cochran-Kim:2004-1}.)

\begin{definition}
  \label{definition:h1-cobordism}
  We say that a relative cobordism $W$ from $M$ to $M'$ is a
  \emph{relative $H_1$-cobordism} if $H_1(M)\cong H_1(W) \cong
  H_1(M')$ under the inclusion-induced maps.
\end{definition}

We will often say ``homology cobordism'' and ``$H_1$-cobordism'',
omitting the word ``relative'', when it is clear that these are
between bordered 3-manifolds from the context.  We remark that in many
cases a cobordism can be surgered, below the middle dimension, to an
$H_1$-cobordism.

The next step is to investigate whether one can eliminate $H_2(W,M)$
by doing surgery; for an $H_1$-cobordism $W$, it is easily seen that
$H_i(W,M)=0$ for $i\ne 2$ and $H_2(W,M)$ is a free abelian group onto
which $H_2(W)$ surjects.  For the convenience of the reader a proof is
given in Lemma~\ref{lemma:H_i(W,M)-for-H_1-cobordism} in
Section~\ref{subsection:basic-properties-H1-cobordism} below.

\begin{definition}
  \label{definition:0-lagrangian}
  Suppose $W$ is an $H_1$-cobordism between bordered 3-manifolds $M$
  and~$M'$.  A subgroup $L\subset H_2(W)$ is called a
  \emph{$0$-lagrangian} if $L$ projects onto a half-rank
  summand of $H_2(W,M)$ isomorphically and the intersection form
  $\lambda_0\colon H_2(W)\times H_2(W)\to \Z$ vanishes on $L\times L$.
\end{definition}

Definition~\ref{definition:0-lagrangian} which uses integral
coefficients is a precursor to the notion of an $n$-lagrangian, which
will be defined in the next subsection in terms of intersection forms
over twisted coefficients.

We remark that one can switch the r\^ole of $M$ and $M'$ in
Definition~\ref{definition:0-lagrangian} as expected, since using
Poincar\'e duality it can be seen that $L \subset H_2(W)$ projects
isomorphically onto a half-rank summand in $H_2(W,M)$ if and only if
$L$ does in $H_2(W,M')$.  We also remark that the following is a
standard fact, which is proven along the lines of the standard surgery
approach.  We give proofs of these two facts in
Section~\ref{subsection:basic-properties-H1-cobordism} below, for the
convenience of the reader.

\begin{proposition}
  \label{proposition:surgery-on-embedded-spherical-lagrangian}
  If an $H_1$-cobordism $W$ between bordered 3-manifolds $M$ and $M'$
  admits a 0-lagrangian $L$ generated by disjoint framed 2-spheres
  embedded in $W$, then $W$ can be surgered to a homology cobordism
  between $M$ and~$M'$.
\end{proposition}

\subsection{Symmetric Whitney tower cobordism of bordered 3-manifolds}

As suggested in
Proposition~\ref{proposition:surgery-on-embedded-spherical-lagrangian}
above, one seeks disjointly embedded framed spheres generating a
0-lagrangian of an $H_1$-cobordism.  As approximations of embeddings,
we recall the notion of a symmetric Whitney tower.

\begin{definition}[{\cite[Definition~7.7]{Cochran-Orr-Teichner:1999-1}}]
  Suppose $S$ is a collection of transverse framed surfaces immersed
  in a 4-manifold~$W$.

  \begin{enumerate}
  \item A \emph{symmetric Whitney tower of height $n$} based on $S$ is
    a sequence $\cC_0,\ldots,\cC_n$ such that $\cC_0 = S$, and for
    $k=1,\ldots,n$, $\cC_{k}$ is a collection of transverse framed
    immersed Whitney disks that pair up all the intersection points of
    $\cC_{k-1}$ and have interior disjoint to surfaces in $\cC_0\cup
    \cdots\cup \cC_{k-1}$.

  \item A \emph{symmetric Whitney tower of height $n.5$} based on $S$
    is a sequence of collections $\cC_0,\ldots,\cC_n$, $\cC_{n+1}$
    such that $\cC_0,\ldots,\cC_{n+1}$ satisfy the defining condition
    of a Whitney tower of height $n+1$ except that the interior of
    $\cC_{n+1}$ is allowed to meet $\cC_n$, while it is still required
    to be disjoint to $\cC_0\cup \cdots\cup \cC_{n-1}$.
  \end{enumerate}
  We call $\cC_k$ the \emph{$k$th stage}, and Whitney disks in $\cC_k$
  are said to be of \emph{height~$k$}.
\end{definition}

Here, intersection points of $\cC_k$ designate both self-intersections
of a surface in $\cC_k$ and intersections of distinct surfaces.  (We
remark that we may assume that no Whitney disk has self-intersections
by ``Whitney tower splitting''
\cite[Section~3.7]{Schneiderman:2006-1}.)  We always assume that
Whitney towers are framed in the sense that for each Whitney disk $D$
that pairs intersections of two sheets, the unique framing on $D$
gives rise to the Whitney section on $\partial D$, which is defined to
be the push-off of $\partial D$ along the tangential direction of one
sheet and along the normal direction of another sheet (avoiding the
tangential direction of~$D$).  As a reference, for example, see
\cite[Section~2.2]{Conant-Teichner-Schneiderman:2012-1} and
\cite[p.~54]{Scorpan:2005-1}.

We remark that if a collection of framed immersed 2-spheres $S_i$
admits a Whitney tower of height $>0$, then it is easily seen that
both the intersection number $\lambda(S_i,S_j)\in \Z[\pi_1(W)]$ and
the self-intersection number $\mu(S_i)\in \Z[\pi_1(W)]/\langle g-\bar
g\rangle$ vanish for all $i$ and~$j$.  (See
\cite[Section~1.7]{Freedman-Quinn:1990-1} for the definition of
$\lambda$ and~$\mu$.)  Consequently the untwisted intersection
$\lambda_0$ automatically vanishes on the~$S_i$.  Also, the converse
is true:

\begin{lemma}
  \label{lemma:intersection-and-0.5-whitney-tower}
  A collection of framed immersed 2-spheres $S_i$ admits a Whitney
  tower of height $0.5$ if and only if $\lambda(S_i,S_j)=0$ and
  $\mu(S_i)=0$ for all $i$ and~$j$.
\end{lemma}

\begin{proof}
  The only if direction has been discussed above.  Conversely, if
  $\lambda$ and $\mu$ vanish on the $S_i$, then the intersection
  points of the $S_i$ can be paired up in such a way that a Whitney
  circle for each pair is null-homotopic.  Applying the Immersion
  Lemma in \cite[p.~13]{Freedman-Quinn:1990-1} and the standard
  boundary twisting operation \cite[p.~16]{Freedman-Quinn:1990-1}, one
  obtains immersed framed Whitney disks, which form a Whitney tower of
  height~$0.5$.
\end{proof}

\begin{definition}
  \label{definition:Whitney-tower-cobordism}
  Suppose $W$ is an $H_1$-cobordism.
  \begin{enumerate}
  \item We say that a submodule $L\subset
    H_2(W;\Z[\pi_1(W)])\cong\pi_2(W)$ is a \emph{framed spherical
      lagrangian} if $L$ projects onto a $0$-lagrangian in $H_2(W)$
    and generated by framed immersed 2-spheres for which $\lambda$
    and~$\mu$ vanish.
  \item We say that $W$ is a \emph{height $h$ Whitney tower cobordism}
    if there is a framed spherical lagrangian generated by framed
    immersed 2-spheres admitting a Whitney tower of height~$h$.  If
    there exists such $W$, we say that $M$ is \emph{height $n.5$
      Whitney tower cobordant} to~$M'$.
  \end{enumerate}
\end{definition}

From Lemma~\ref{lemma:intersection-and-0.5-whitney-tower}, the
following is immediate: there is a framed spherical lagrangian if and
only if there exist immersed 2-spheres that generate a 0-lagrangian in
$H_2(W)$ and support a Whitney tower of height~$0.5$.

\subsection{Solvable cobordism of bordered 3-manifolds}

Following the ideas of \cite[Definitions 8.5, 8.7 and Theorems 8.6,
8.8]{Cochran-Orr-Teichner:1999-1}, we relate Whitney towers to
lagrangians admitting \emph{duals}.  Later this will enable us to
obtain amenable $L^2$-signature invariant obstructions.  Our
definition below, which is for bordered 3-manifolds, is also similar
to the notion of $h$-cylinders considered by Cochran and Kim for
\emph{closed} 3-manifolds with first Betti number one
\cite[Definition~2.1]{Cochran-Kim:2004-1}.

We fix some notation.  For a group $G$, $G^{(n)}$ denotes the $n$th
derived subgroup defined by $G^{(0)}=G$, $G^{(n+1)} = [G^{(n)},
G^{(n)}]$.
% An immersed surface $N$ in a 4-manifold $W$ with $\pi=\pi_1(W)$ is
% called an \emph{$n$-surface} if $N$ represents an element in
% $H_2(W;\Z[\pi/\pi^{(n)}])$, i.e., $N$ lifts to the regular cover
% with fundamental group~$\pi^{(n)}$.
For a 4-manifold $W$ with $\pi=\pi_1(W)$, let
\[
\lambda_n\colon H_2(W;\Z[\pi/\pi^{(n)}])\times H_2(W;\Z[\pi/\pi^{(n)}]) \to
\Z[\pi/\pi^{(n)}]
\]
be the intersection form.  We say that a closed surface immersed in
$W$ is an \emph{$n$-surface} if it represents an element in
$H_2(W;\Z[\pi/\pi^{(n)}])$, namely it lifts to the regular cover of
$W$ with fundamental group~$\pi^{(n)}$.

\begin{definition}
  \label{definition:lagrangian-dual-solvable-cobordism}
  Suppose $W$ is an $H_1$-cobordism between bordered 3-manifolds $M$
  and~$M'$ with $\pi=\pi_1(W)$. Let $m=\frac 12 \rank H_2(W,M)$.
  \begin{enumerate}
  \item A submodule $L\subset H_2(W;\Z[\pi/\pi^{(n)}])$ is an
    \emph{$n$-lagrangian} if $L$ projects onto a $0$-lagrangian for
    $H_2(W,M)$ and $\lambda_n$ vanishes on~$L$.
  \item For an $n$-lagrangian ($n\ge k$) or a framed spherical
    lagrangian $L$, homology classes $d_1,\ldots,d_m\in
    H_2(W;\Z[\pi/\pi^{(k)}])$ are \emph{$k$-duals} of $L$ if $L$ is
    generated by $\ell_1,\ldots,\ell_m\in L$ whose projections $\ell'_i
    \in H_2(W;\Z[\pi/\pi^{(k)}])$ satisfy
    $\lambda_k(\ell'_i,d_j)=\delta_{ij}$.
  \item The 4-manifold $W$ is an \emph{$n.5$-solvable cobordism}
    (resp.\ \emph{$n$-solvable cobordism}) if it has an
    $(n+1)$-lagrangian (resp.\ $n$-lagrangian) with $n$-duals.  If
    there exists an $h$-solvable cobordism from $M$ to $M'$, we say
    that $M$ is \emph{$h$-solvably cobordant} to~$M'$.
  \end{enumerate}
\end{definition}

We remark that the notion of an $h$-solvable cobordism is a relative analogue of an
$h$-solution introduced in \cite{Cochran-Orr-Teichner:1999-1}, as
mentioned in the introduction.

\begin{theorem}
  \label{theorem:whitney-tower-and-solvability}
  Suppose $M$ and $M'$ are bordered 3-manifolds.  Then for the
  following statements, $(1) \Rightarrow (2) \Rightarrow (3)$ holds:
  \begin{enumerate}
  \item $M$ and $M'$ are height $n.5$ Whitney tower cobordant.
  \item There is an $H_1$-cobordism between $M$ and $M'$ which has a
    framed spherical lagrangian admitting $n$-duals.
  \item $M$ and $M'$ are $n.5$-solvably cobordant.
  \end{enumerate}
\end{theorem}

\begin{proof}
  First, $(2)$ implies $(3)$ since a framed spherical lagrangian is an
  $(n+1)$-lagrangian.  The implication $(1) \Rightarrow (2)$ is proven
  by an argument similar to \cite[Proof of
  Theorem~8.4]{Cochran-Orr-Teichner:1999-1} (see the part entitled ``
  the induction step $r \mapsto r-1$'').  The point is that for our
  purpose we do not need the assumption of
  \cite{Cochran-Orr-Teichner:1999-1} that the concerned 4-manifolds
  are spin.  We give details for concreteness and for the reader's
  convenience.

  Suppose $W$ is an $H_1$-cobordism from $M$ to $M'$, and a spherical
  lagrangian for $W$ is generated by framed immersed 2-spheres
  $\ell_i$ which support a Whitney tower of height $h$ and admit
  $r$-duals $d_j$.  We will show, if $h\ge 1.5$, there is an
  $H_1$-cobordism from $M$ to $M'$ with a spherical lagrangian
  generated by framed immersed 2-spheres which support a Whitney tower
  of height $h-1$ and admit $(r+1)$-duals.  From this our conclusion
  follows by an induction on $r$ starting from $(h,r)=(n.5,0)$; one
  can start the induction since a spherical lagrangian always admits
  $0$-duals by Lemma~\ref{lemma:langrangian-always-has-duals} stated
  and proved later (see
  Section~\ref{subsection:basic-properties-H1-cobordism}).

  The claim is proven as follows.  Let $\cC$ be the given Whitney
  tower of height~$h$.  By pushing down intersections as in
  \cite[Section~2.5]{Freedman-Quinn:1990-1} we may assume that each
  $d_j$ does not meet height $>0$ part of $\cC$.  By tubing if
  necessary, one may assume the geometric intersection of $\ell_i$ and
  $d_j$ is precisely~$\delta_{ij}$.  Denote the collection of the
  Whitney circles pairing intersections of the $\ell_i$ by
  $\{\alpha_k\}$, and let $\Delta_k$ be the height $1$ Whitney disk
  bounded by~$\alpha_k$.  Choose one of the two intersection points
  lying on $\alpha_k$, and around it, choose a linking torus~$T_k$
  (see \cite[p.~12]{Freedman-Quinn:1990-1}) which is disjoint from the
  $\ell_i$ and~$d_j$.  We may assume that $T_k$ intersects $\cC$ at a
  single point on~$\Delta_k$.  Let $x_k$ and $y_k$ be the standard
  basis curves on $T_k$ based at $T_k\cap\cC$.  Since $x_k$ and $y_k$
  are meridians of some of the $\ell_i$, these are conjugate to
  elements of~$\pi_1(d_i)^{(1)}$.  Since $\pi_1(d_i) \subset
  \pi_1(W)^{(r)}$, it follows that $T_k$ is an $(r+1)$-surface in~$W$.

  Now do surgery on $W$ along pushoffs of the $\alpha_k$ taken along
  the $\Delta_k$ direction to a new 4-manifold, in which we have
  framed embedded 2-disks $b_k$ bounded by~$\alpha_k$.  By Whitney
  moves along the $b_k$, isotope the $\ell_i$ to disjointly embedded
  framed 2-spheres.  Doing surgery along these 2-spheres, we obtain
  another new 4-manifold, say~$W'$.  The framed immersed 2-spheres
  $\ell_k' := \Delta_k \cup_\partial b_k$ together with height $\ge 2$
  Whitney disks of $\cC$ form a Whitney tower of height~$h-1$.  Since
  $h-1\ge 0.5$, the intersection $\lambda$ and self-intersection $\mu$
  vanish on the~$\ell_k'$.  Direct computation of the rank of $H_2$
  shows that the $\ell_k'$ form a framed spherical lagrangian
  for~$W'$.  Since the geometric intersection of $\ell_k'$ and $T_l$
  is precisely $\delta_{kl}$, the $T_k$ are $(r+1)$-duals.
\end{proof}
  
\begin{remark}
  \label{remark:comparison-with-COT}
  In \cite{Cochran-Orr-Teichner:1999-1} they make an additional
  assumption that the concerned 4-manifolds are spin.  If one adds the
  similar spin condition and self-intersection vanishing condition to
  our definitions, then the arguments in
  \cite{Cochran-Orr-Teichner:1999-1} can be carried out to show that
  all the statements (1), (2), and (3) in
  Theorem~\ref{theorem:whitney-tower-and-solvability} are equivalent.
  A key technical point is that the spin assumption implies that
  $k$-duals are represented by surfaces which are automatically
  \emph{framed}.
\end{remark}

\begin{remark}
  \label{remark:whitney-tower-and-solvability-n-case}
  One can also show the following: \emph{if $M$ and $M'$ are height
    $n$ Whitney tower cobordant, then $M$ and $M'$ are $n$-solvably
    cobordant.}  Indeed, applying the induction as in the above proof,
  one obtains a spherical lagrangian supporting a height one Whitney
  tower together with $(n-1)$-duals.  Applying the induction argument
  once more, one now obtains framed immersed spheres $\ell_k'$ and the
  tori $T_k$ which are $n$-surfaces, but now the $\ell_k$ may have
  non-vanishing intersection $\lambda$.  Nonetheless, since the tori $T_k$
  are mutually disjoint, one sees that the $T_k$ form an
  $n$-lagrangian and the $\ell_k$ are their $n$-duals.
\end{remark}

\subsection{Symmetric Whitney tower concordance and grope concordance of links}

Recall that two $m$-component links $L$ and $L'$ in $S^3$ are
\emph{concordant} if there is a collection of $m$ disjoint cylinders
properly embedded in $S^3\times [0,1]$ joining the corresponding
components of $L\times 0$ and $-L'\times 1$.  We always assume links
are ordered.

It is natural to think of immersed cylinders supporting Whitney
towers, as an approximation of honest concordance.

\begin{definition}
  \label{definition:whitney-tower-concordance}
  Two $m$-component links $L$ and $L'$ in $S^3$ are \emph{height $h$
    (symmetric) Whitney tower concordant} if there is a collection of
  transverse framed cylinders $C_i$ ($i=1,\ldots,m$) immersed in
  $S^3\times [0,1]$ which joins the 0-framed $i$th components of
  $-L\times 0$ and $L'\times 1$, and there is a Whitney tower of
  height $h-1$ based on the~$C_i$.
\end{definition}

Note that ``height $h-1$'' is not a typo.  This is because the
following convention: the immersed annuli $C_i$ are said to be the
height one part of the Whitney tower concordance.  ($-L\times 0 \cup
L'\times 1$ is said to be the height zero part; see also
\cite[Definition~7.7]{Cochran-Orr-Teichner:1999-1}.)

The following is a Whitney tower analogue of the fact that the
exteriors of concordant links are, as bordered 3-manifolds, relatively
homology cobordant.

\begin{theorem}
  \label{theorem:link-whitney-tower-concobordism}
  If two links are height $h+2$ Whitney tower concordant, then their
  exteriors are height $h$ Whitney tower cobordant, as bordered
  3-manifolds.
\end{theorem}

The proof is parallel to that of
\cite[Theorem~8.12]{Cochran-Orr-Teichner:1999-1}.  Details are
omitted.

% \begin{proof}
%   Suppose $\cC$ is a Whitney tower concordance between two links $L$
%   and $L'$ with exteriors $M$ and~$M'$.  Let $\{\alpha_k\}$ be the
%   collection of Whitney circles joining intersections of the base
%   immersed annuli of~$\cC$.  Let $\Delta_k$ be the next stage Whitney
%   disk with boundary~$\alpha_k$.  By surgery along parallel copies of
%   the $\alpha_k$, one obtains a 4-manifold, say $V$, together with
%   embedded 2-disks $b_k$ bounded by the~$\alpha_k$.  Whitney moves
%   along the $b_k$ isotopes the base cylinders to embedded cylinders
%   in~$V$.  Let $W$ be the exterior of these embedded cylinders in~$V$.
%   Then $\partial W$ is equal to $M\cup_\partial M'$.  By a
%   straightforward $H_2$ computation, one sees that the framed immersed
%   2-spheres $S_k:=\Delta_k \cup_\partial b_k$ form a framed spherical
%   lagrangian.  The upper part of the Whitney tower $\cC$ becomes a
%   Whitney tower of height $h-2$ based on the~$S_k$.
% \end{proof}

Another well-known notion generalizing link concordance is grope
concordance.  We consider \emph{symmetric} gropes only,
which have a \emph{height}.  For the reader's convenience we give
definitions below.

\begin{definition}
  \label{definition:grope}
  Let $n$ be a nonnegative integer.  A \emph{grope of height $n$}
  based on a circle $\gamma$ is defined inductively as follows.  A
  grope of height $0$ based on $\gamma$ is $\gamma$ itself.  A
  \emph{grope of height $n$ based on $\gamma$} consists of a genus $g$
  oriented surface $S$ bounded by $\gamma$, and $2g$ symmetric gropes
  of height $n-1$ based on a circle which is attached to $S$ along
  $2g$ simple closed curves $a_1,\ldots,a_g, b_1,\ldots,b_g$ on $S$
  which form a symplectic basis (that is, the geometric intersections
  are given by $a_i\cdot a_j=0=b_i\cdot b_j$, $a_i\cdot
  b_j=\delta_{ij}$).  A \emph{grope of height $n.5$ based on $\gamma$}
  consists of a genus $g$ oriented surface $S$ bounded by $\gamma$ and
  $g$ symmetric gropes of height $n$ based on a circle attached to $S$
  along the half basis curves $a_i$, and $g$ symmetric gropes of
  height $n-1$ based on a circle attached to $S$ along the remaining
  curves~$b_j$.  The surface $S$ above is called the \emph{1st stage}
  of the grope.

  An \emph{annular grope of height $h$} is defined by replacing $S$
  above with a genus $g$ oriented surface with two boundary
  components.

  A grope embeds into~$\R^3$ in a standard way, and then into $\R^4$
  via $\R^3\subset \R^4$.  A \emph{framed embedding of a grope} in a
  4-manifold is an embedding of a regular neighborhood of its standard
  embedding in~$\R^4$.

\end{definition}

\begin{definition}
  \label{definition:grope-concordance}
  Two $m$-component links $L$ and $L'$ in $S^3$ are \emph{height $h$
    grope concordant} if there are $m$ framed annular
  gropes $G_i$ ($i=1,\ldots,m$) disjointly embedded in
  $S^3\times[0,1]$ which are cobounded by the zero-framed $i$th
  components of $-L\times 0$ and $L\times 1$.
\end{definition}

Schneiderman showed that if a knot is height $h$ grope concordant to
the unknot, then the knot is height $h$ Whitney tower concordant to
the unknot~\cite[Corollary~2]{Schneiderman:2006-1}.  One can verify
that his proof in \cite[Section~6]{Schneiderman:2006-1} can be carried
out for the case of grope concordance between links:

\begin{theorem}[{Link version of
    \cite[Corollary~2]{Schneiderman:2006-1}}]
  \label{theorem:grope-implies-whitney-tower}
  Two links are height $h$ Whitney tower cobordant if they are height
  $h$ grope concordant.
\end{theorem}

We remark that the converse of
Theorem~\ref{theorem:grope-implies-whitney-tower} is unknown, while
the asymmetric analogue (obtained by replacing ``height'' with
``order'') and its converse are both true due
to~\cite{Schneiderman:2006-1}.

Applying Theorem~\ref{theorem:link-whitney-tower-concobordism} and
Theorem~\ref{theorem:whitney-tower-and-solvability}, we obtain the
following result immediately:

\begin{corollary}
  \label{corollary:link-whitney-tower-solvability}
  If two links are height $n+2.5$ Whitney tower concordant or height
  $n+2.5$ grope concordant, then their exteriors are $n.5$-solvably
  cobordant, as bordered 3-manifolds.
\end{corollary}

\begin{remark}
  Since $S^3\times[0,1]$ is spin, one can strengthen the conclusions
  of Theorem~\ref{theorem:link-whitney-tower-concobordism} and
  Corollary \ref{corollary:link-whitney-tower-solvability}, using
  Remark~\ref{remark:comparison-with-COT}: there exist a spin height
  $h$ Whitney tower cobordism and a spin $h$-solvable cobordism
  between the exteriors.
\end{remark}

\begin{remark}
  \label{remark:topological-vs-smooth}
  Everything in this paper can be carried out in both topological
  (assuming submanifolds are locally flat) and smooth category.
  Indeed, regarding our setup given in this section, one can see that
  the topological and smooth cases are equivalent in the following
  sense.  (1) \emph{Two bordered 3-manifolds are topologically
    $h$-solvably cobordant if and only if these are smoothly
    $h$-solvably cobordant.} (2) \emph{Two links are topologically
    height $h$ Whitney tower (resp.\ grope) concordant if and only if
    these are smoothly height $h$ Whitney tower (resp.\ grope)
    concordant.}  We give only a brief outline of the proof, omitting
  details.  (1)~From a topological $h$-solvable cobordism, one obtains
  a smooth 4-manifold by taking connected sum with
  $|E_8|\#8\overline{\C P^2}$ if the Kirby-Siebenmann invariant is
  nonzero, and then with copies of $S^2\times S^2$, where $|E_8|$ is
  Freedman's manifold with intersection form
  $E_8$~\cite[Chapter~10]{Freedman-Quinn:1990-1}.  One can see that
  the result is an $h$-solvable cobordism, using that
  $E_8\#8\overline{\C P^2}$ and $S^2\times S^2$ are simply connected
  and have a 0-lagrangian with 0-duals.  (2)~Appealing to Quinn's
  smoothing theorem \cite[\textsection 8.1]{Freedman-Quinn:1990-1}, a
  topological height $h$ Whitney tower concordance in $S^3\times[0,1]$
  deforms to a smoothly immersed 2-complex in such a way that new
  intersections are paired up by topological Whitney disks.  It turns
  out that similar smoothing arguments repeatedly applied to the
  topological Whitney disks give a smooth Whitney tower concordance of
  desired height.  For a topological grope concordance, the above
  smoothing arguments combined with Schneiderman's method
  \cite[Theorem~5]{Schneiderman:2006-1} that converts Whitney towers
  to gropes give a desired smooth grope concordance.
\end{remark}

% \begin{remark}
%   \leavevmode\Nopagebreak\todo{Determine if this remark is actually needed.}
%   \begin{enumerate}
%   \item In case of knots, our setup reduces to that of
%     Cochran-Orr-Teichner as follows: one sees that for two knots $K$
%     and $K'$ with exteriors $E$ and $E'$, the 3-manifold
%     $E\cup_\partial E'$ is equal to the zero-surgery manifold
%     $M(K\#-K')$ of the connected sum, and a spin $h$-solvable
%     cobordism (resp.\ a spin $H_1$-cobordism) between $E$ and $E'$ is
%     none more than an $h$-solution (resp. an $H_1$-bordism) of
%     $M(K\#-K')$ in the sense of~\cite{Cochran-Orr-Teichner:1999-1}.
%   \item For the link case, our setup provides more generality. For
%     example, the approaches of \cite{Cochran-Orr-Teichner:1999-1} and
%     subsequent works do not apply to the study of concordance between
%     two links with nontrivial linking number.
%   \end{enumerate}
% \end{remark}

\subsection{Basic properties of an $H_1$-cobordism}
\label{subsection:basic-properties-H1-cobordism}

In this section we give proofs of a few basic observations used in the
earlier parts of this section, for completeness and for the
convenience of readers.

\begin{lemma}
  \label{lemma:H_i(W,M)-for-H_1-cobordism}
  Suppose $W$ is an $H_1$-cobordism between bordered 3-manifolds $M$
  and~$M'$.  Then the following hold:
  \begin{enumerate}
  \item $H_i(W,M)=0=H_i(W,M')$ for $i\ne 2$.
  \item $W$ is a homology cobordism if and only if $H_2(W,M)=0$.
  \item The map $H_2(W)\to H_2(W,M)$ is surjective, and consequently
    $H_2(W,M)=\Coker\{H_2(M)\to H_2(W)\}$.  Similarly for~$M'$.
  \item $H_2(W,M)$ and $H_2(W,M')$ are torsion-free abelian groups of
    the same rank.
 \end{enumerate}
\end{lemma}

% Based on (3) above, we often regard an element of $H_2(W)$ as its
% image in $H_2(W,M)$.

\begin{proof}
  (3) and (1)${}_{i<2}$ follow from the long exact sequence and the
  $H_1$-cobordism condition.  By the
  Poincar\'e duality for relative cobordism and the universal
  coefficient theorem, we have $H_2(W,M)=\Hom(H_2(W,M'),\Z)$.  From this (4)
  follows.  Also it implies (1)${}_{i>2}$ since
  $H_i(W,M)=\Hom(H_{4-i}(W,M'),\Z)=0$ for $i>2$.  Now (2) follows
  from~(1).
\end{proof}

\begin{lemma}
  \label{lemma:langrangian-always-has-duals}
  Suppose $W$ is an $H_1$-cobordism between bordered 3-manifolds $M$
  and $M'$ and $\ell_1,\ldots,\ell_m \in H_2(W)$.
  \begin{enumerate}
  \item If the $\ell_i$ form a basis of a summand of $H_2(W,M)$, then
    there are $d_1,\ldots,d_m\in H_2(W)$ satisfying
    $\lambda_0(\ell_i,d_j)=\delta_{ij}$.
  \item If $\lambda_0(\ell_i,\ell_j)=0$ for any $i,j$ and there exist
    $d_1,\ldots,d_m\in H_2(W)$ satisfying
    $\lambda_0(\ell_i,d_j)=\delta_{ij}$, then $\{\ell_i, d_j\}$ is a
    basis of a summand of $H_2(W,M)$. Consequently, if in addition
    $m=\frac 12 \rank H_2(W,M)$, then $\{\ell_i\}$ generates a
    0-lagrangian and $\{\ell_i, d_j\}$ is a basis of $H_2(W,M)$.
  % \item If the $\ell_i$ form a basis of a $0$-lagrangian for
  %   $H_2(W,M)$, then the $\ell_i$ together with (the images of) the
  %   duals $d_j$ in (1) form a basis of $H_2(W,M)$.
  \end{enumerate}
\end{lemma}

We remark there are several useful consequences of
Lemma~\ref{lemma:langrangian-always-has-duals}.  First, it follows
that any $0$-lagrangian has $0$-duals.  From this
Proposition~\ref{proposition:surgery-on-embedded-spherical-lagrangian}
follows, since surgery on framed embedded spheres generating a
$0$-lagrangian eliminates the 0-duals as well.  Finally, in
Definition~\ref{definition:0-lagrangian}, the r\^oles of $M$ and $M'$
can be switched.

\begin{proof}%[Proof of Lemma~\ref{lemma:langrangian-always-has-duals}]
  (1) Let $PD\colon H_2(W,M) \to \Hom(H_2(W,M'),\Z)$ be the relative
  Poincar\'e duality isomorphism.  Extend the classes of the $\ell_i$
  to a basis of $H_2(W,M)$ and choose a basis of $H_2(W,M') =
  \Hom(\Hom(H_2(W,M'),\Z),\Z)$ dual to $PD(\ell_i)$.  Since $H_2(W)\to
  H_2(W,M')$ is surjective, the dual basis elements are represented by
  some $d_i \in H_2(W)$.  By definition, viewing $\lambda_0$ as
  $H_2(W) \to \Hom(H_2(W),\Z)$, $\lambda_0$ is the composition of
  inclusion-induced maps with the isomorphism $PD\colon H_2(W,M) \to
  \Hom(H_2(W,M'),\Z)$.  Thus
  $\lambda_0(\ell_i,d_j)=PD(\ell_i)(d_j)=\delta_{ij}$ as desired.

  (2) Let $A=\Z^{2m}$ which is endowed with the standard basis
  $\{e_i\}$, and $f\colon A \to H_2(W,M)$ and $g\colon A \to
  H_2(W,M')$ be the maps sending $e_i$ to the image of $\ell_i$ for
  $i\le m$ and to the image of $d_{i-m}$ for $i > m$.  Then the
  composition
  \[
  A \xrightarrow{f} H_2(W,M)\xrightarrow{PD}\Hom(H_2(W,M'),\Z)
  \xrightarrow{g^*} \Hom(A,\Z)
  \]
  is represented by the block matrix $\big[\begin{smallmatrix}0 & I \\
    I & * \end{smallmatrix}\big]$.  In particular it is an
  isomorphism.  Since all the terms in the composition are free
  abelian groups of rank $2m$, it follows that $f$ is an isomorphism.
  %
  % If $\{\ell_i\}$ generates a 0-lagrangian $L \subset H_2(W,M)$, then
  % using the intersection pairing $\bar\lambda\colon H_2(W,M)\times
  % H_2(W)\to \Z$, it is shown that the $\{\ell_i,d_j\}$ is a basis of
  % $H_2(W,M)$.  For concreteness, we give details: first one sees that
  % $\{\ell_i,d_j\}$ is linearly independent in $H_2(W,M)$ from the fact
  % that if $x=\sum a_i \ell_i +\sum b_j d_j$ in $H_2(W,M)$, then we
  % have $b_j=\bar\lambda(x,\ell_j)$ and $a_i=\bar\lambda(x,d_i)-\sum_j
  % \bar\lambda(x,\ell_j)\bar\lambda(d_j,d_i)$.  By dimension counting,
  % $\{\ell_i,d_j\}$ is a basis of $H_2(W,M)\otimes \Q$.  Since
  % $H_2(W,M)$ is torsion free, it follows that $x=0$ in $H_2(W,M)$ if
  % and only if $\bar\lambda(x,d_i)=0=\bar\lambda(x,\ell_j)$ for all
  % $i,j$.  Therefore for any $x\in H_2(W,M)$, $x-(\sum a_i \ell_i +\sum
  % b_j d_j)=0$ where $a_i$, $b_j$ are given by the above formula.  This
  % shows that $\{\ell_i,d_j\}$ spans $H_2(W,M)$ over~$\Z$.
\end{proof}

% The following is another immediate consequence of
% Lemma~\ref{lemma:langrangian-always-has-duals}.

% Note that if $W$ is spin, then any 2-spheres embedded in $W$ are
% automatically framed so that one can do surgery.

\section{Amenable signature theorem for Whitney towers}
\label{section:amenable-signature-theorem}

We denote by $\N G$ the group von Neumann algebra of a discrete
countable group~$G$.  For a finitely generated $\N G$-module $M$, the
\emph{$L^2$-dimension} $\ldim M \in \R_{\ge 0}$ can be defined.  For
more information on $\N G$ and the $L^2$-dimension, see L\"uck's book
\cite{Lueck:2002-1} and his paper~\cite{Lueck:1998-1}.  Also
\cite[Section~3.1]{Cha:2010-01} gives a quick summary of the
definition and properties of the $L^2$-dimension which are useful for
our purpose.

The algebra $\N G$ is endowed with the natural homomorphism $\Z G \to
\N G$, so that one can view $\N G$ as a $\N G$-$\Z G$ bimodule.  For a
finite CW pair $(X,A)$ endowed with $\pi_1(X) \to G$, its cellular
homology $H_*(X,A;\N G)$ with coefficients in $\N G$ is defined to be
the homology of the chain complex $\N G\otimesover{\Z G} C_*(X,A;\Z
G)$.  We denote the \emph{$L^2$-Betti number} by
\[
\bt_i(X,A;\N G) = \dimt H_i(X,A;\N G).
\]
When the choice of $\pi_1(X)\to G$ is clearly understood,
$\bt_i(X,A;\N G)$ is denoted by $\bt_i(X,A)$.

We denote by $b_i(X,A;R)$ the ordinary Betti number $\dim_R
H_i(X,A;R)$ for a field $R$, particularly for $R=\Q$ or~$\Z/p$.  We
write $b_i(X,A)=b_i(X,A;\Q)$ as usual.

For a closed 3-manifold $M$ and a homomorphism $\phi\colon\pi_1(M)\to
G$ into a discrete countable group $G$, we denote the von
Neumann-Cheeger-Gromov $\rho$-invariant by $\rhot(M,\phi)\in \R$.
See, for example, \cite[Section~5]{Cochran-Orr-Teichner:1999-1} as
well as
\cite{Chang-Weinberger:2003-1,Harvey:2006-1,Cha:2006-1,Cha-Orr:2009-01}
as references providing definitions and properties of $\rhot(M,\phi)$
useful for our purpose.

\begin{definition}
  \label{definition:amenability-strebel-D(R)}
  \leavevmode\Nopagebreak
  \begin{enumerate}
  \item A discrete group $G$ is \emph{amenable} if there is a finitely
    additive measure on $G$ which is invariant under the left
    multiplication.
  \item For a commutative ring $R$ with unity, a group $G$ lies in
    \emph{Strebel's class $D(R)$} if a homomorphism $\alpha\colon P\to
    Q$ between projective $RG$-modules is injective whenever
    $1_R\otimes_{RG}\alpha\colon R\otimes_{RG}P \to R\otimes_{RG}Q$ is
    injective~\cite{Strebel:1974-1}.
  \end{enumerate}
\end{definition}

The main result of this section is stated below.  

\begin{theorem}[Amenable Signature Theorem for solvable cobordism]
\label{theorem:amenable-signature-for-solvable-cobordism}
Suppose $W$ is a relative cobordism between two bordered 3-manifolds
$M$ and~$M'$, $G$ is an amenable group lying in $D(R)$, $R=\Z/p$ or
$\Q$, and $G^{(n+1)}=\{e\}$.  Suppose $\phi\colon \pi_1(M\cup_\partial
M') \to G$ extends to $\pi_1(W)$, and either one of the following
conditions holds:
\begin{enumerate}
\item[(I)] $W$ is an $n.5$-solvable cobordism and $\bt_1(M;\N G)=0$.
\item[(II)] $W$ is an $n.5$-solvable cobordism,
  $|\phi(\pi_1(M))|=\infty$, and
  \[
  \bt_1(M\cup_\partial M'; \N G) \ge b_1(M;R)+b_2(M;R)+b_3(M;R)-1.
  \]
\item[(III)] $W$ is an $(n+1)$-solvable cobordism.
\end{enumerate}
Then $\rhot(M\cup_\partial M',\phi)=0$.
\end{theorem}

\begin{remark}
  \leavevmode\Nopagebreak
  \begin{enumerate}
  \item The class of amenable groups in $D(R)$ is large.  For example
    see \cite{Cha-Orr:2009-01}, especially Lemma~6.8 and the
    discussion above it.  As a special case,
    Theorem~\ref{theorem:amenable-signature-for-solvable-cobordism}
    can be applied when $G$ is a PTFA group satisfying
    $G^{(n+1)}=\{e\}$.

  \item Case (I) provides a new interesting case.
    Section~\ref{subsection:vanishing-of-b^2_1} gives some useful
    instances to which case (I) applies.  In particular case (I) will
    be used to provide new applications to links with nonvanishing
    linking number in this paper.  See
    Section~\ref{section:concordance-to-hopf-link}.  Cases (II) and
    (III) are closely related to previously known results.  See
    Section~\ref{subsection:relationship-with-previous-results}.
    Further applications of (II) and (III) will be given in other
    papers.

  \item The assumption $|\phi(\pi_1(M))|=\infty$ in case (II) is not
    severe, since in many cases we are interested in infinite covers
    of~$M$ to extract deeper information.
  \end{enumerate}
\end{remark}

Recall from Corollary~\ref{corollary:link-whitney-tower-solvability}
that if two links are height $h+2$ Whitney tower (or grope)
concordant, then their exteriors are $h$-solvable cobordant.
Therefore
Theorem~\ref{theorem:amenable-signature-for-solvable-cobordism} also
obstructs height $n+2.5$ and $n+3$ Whitney tower (and grope)
concordance of links.

% \begin{remark}
%   Several prior known results on $L^2$-signatures obstructing the
%   existence of Whitney towers are indeed special cases of
%   Theorem~\ref{theorem:amenable-signature-for-solvable-cobordism}:
%   \begin{enumerate}
%   \item The amenable $L^2$-signature obstruction to a knot being
%     $n.5$-solvable \cite[Theorem~3.2]{Cha:2010-01} is a consequence of
%     Theorem~\ref{theorem:amenable-signature-for-solvable-cobordism}.
%   \item In particular, Cochran-Orr-Teichner's
%     poly-torsion-free-abelian signature obstruction to a knot being
%     $n.5$-solvable \cite[Theorem~4.2]{Cochran-Orr-Teichner:1999-1} is
%     a consequence of
%     Theorem~\ref{theorem:amenable-signature-for-solvable-cobordism}.
%   \item{} [Harvey] and [Cochran-Harvey-Leidy] for links (?)
%   \end{enumerate}
% \end{remark}

% \begin{remark}
%   The key assumption in
%   Theorem~\ref{theorem:amenable-signature-for-solvable-cobordism} is
%   weaker than those in other similar results: we only assume that
%   there is an $(n+1)$-lagrangian, while prior results are stated
%   assuming $n$-solvablity, namely that there are $n$-duals in
%   addition.
% \end{remark}

The proof of
Theorem~\ref{theorem:amenable-signature-for-solvable-cobordism} is
given in Section~\ref{subsection:proof-amenable-signature-theorem}.
Readers more interested in its applications and relationship with
previously known results may skip the proof and proceed to
Sections~\ref{subsection:vanishing-of-b^2_1},
\ref{subsection:relationship-with-previous-results}, and then to
Section~\ref{section:concordance-to-hopf-link}.

\subsection{Proof of Amenable Signature
  Theorem~\ref{theorem:amenable-signature-for-solvable-cobordism}}
\label{subsection:proof-amenable-signature-theorem}

To prove
Theorem~\ref{theorem:amenable-signature-for-solvable-cobordism}, we
need estimations of various $L^2$-dimensions.  One of the primary
ingredients is the following result which appeared
in~\cite{Cha:2010-01}:

\begin{theorem}[{\cite[Theorem~3.11]{Cha:2010-01}}]
  \label{theorem:L2-dim-estimate-of-NG-homology}
  \leavevmode\Nopagebreak
  \begin{enumerate}
  \item Suppose $G$ is amenable and in $D(R)$ with $R=\Q$ or $\Z/p$,
    and $C_*$ is a projective chain complex over $\Z G$ with $C_n$
    finitely generated.  Then we have
    \[
    \ldim H_n(\N G \otimesover{\Z G}C_*) \le \dim_R
    H_n(R\otimesover{\Z G}C_*).
    \]
  \item In addition, if $\{x_i\}_{i\in I}$ is a collection of
    $n$-cycles in $C_n$, then for the submodules $H\subset H_n(\N G
    \otimes C_*)$ and $\overline H \subset H_n(R\otimes C_*)$
    generated by $\{[1_{\N G}\otimes x_i]\}_{i\in I}$ and
    $\{[1_{R} \otimes x_i]\}_{i\in I}$, respectively, we have
    \[
    \ldim H_n(\N G \otimesover{\Z G} C_*)-\ldim H \le \dim_R
    H_n(R\otimesover{\Z G} C_*) - \dim_R \overline{H}.
    \]
  \end{enumerate}
\end{theorem}

Lemma~\ref{lemma:L2-betti-numbers-of-H_1-cobordism} below states
various Betti number observations for an $H_1$-cobordism.  We remark
that only Lemma~\ref{lemma:L2-betti-numbers-of-H_1-cobordism} (1), (2)
are used in the proof of Amenable Signature
Theorem~\ref{theorem:amenable-signature-for-solvable-cobordism} (I)
and (III).  Lemma~\ref{lemma:L2-betti-numbers-of-H_1-cobordism}
(3)--(7) are used in the proof of case (II).

\begin{lemma}
  \label{lemma:L2-betti-numbers-of-H_1-cobordism}
  Suppose $W$ is a relative $H_1$-cobordism between $M$ and $M'$,
  $R=\Q$ or $\Z/p$, and $\phi\colon \pi_1(W) \to G$ is a homomorphism
  into an amenable group $G$ in $D(R)$.  Then the following hold:
  \begin{enumerate}
  \item $\bt_i(W,M) = 0$ for $i\ne 2$.
  \item $\bt_2(W,M)=b_2(W,M)=b_2(W,M;R)$.
  \item $\bt_0(W,\partial W)=0=\bt_4(W)$.
  \item $\bt_1(W,\partial W)=0=\bt_3(W)$ if either $\partial M\ne
    \emptyset$ or $\Im\{\pi_1(M) \to \pi_1(W)\to G\}$ is infinite.
  \item $\bt_4(W,\partial W)=0=\bt_0(W)$ if $\Im\{\pi_1(W)\to G\}$ is
    infinite.
  \item $b_1(W,\partial W)=b_3(W)=b_3(M)$.
    % =\begin{cases}
    %   1 &\text{if $M$ is closed,}\\
    %   0 &\text{otherwise.}
    % \end{cases}_{\mathstrut}^{\mathstrut}
  \item $b_2(W)=b_2(M)+b_2(W,M)$ and $b_2(W;R)=b_2(M;R)+b_2(W,M;R)$.
 \end{enumerate}
\end{lemma}

\begin{proof}
  Recall that $W$, $M$, $M'$ are all assumed to be connected by our
  convention.

  (1) Applying
  Theorem~\ref{theorem:L2-dim-estimate-of-NG-homology}~(1) to the
  chain complex $C_*(W,M;\Z G)$, it follows that $\bt_i(W,M)\le
  b_i(W,M;R)$.  Thus $\bt_i(W,M;\N G)=0$ for $i\ne 2$ since
  $b_i(W,M;R)=0$ for $i\ne 2$ by
  Lemma~\ref{lemma:H_i(W,M)-for-H_1-cobordism} (1) and an easy
  application of the universal coefficient theorem.

  (2) Note that $b_2(W,M)=b_2(W,M;R)$ by
  Lemma~\ref{lemma:H_i(W,M)-for-H_1-cobordism} (1), (4) and the
  universal coefficient theorem.  Since the Euler characteristic of
  $(W,M)$ can be computed using either $b_i(-)$ or $\bt_i(-;\N G)$,
  from (1) it follows that $\bt_2(W,M;\N G) = b_2(W,M)$.

  (3) Since $W$ is connected and $\partial W$ is nonempty, we may
  assume that there is no 0-cell in the CW complex structure of the
  pair $(W,\partial W)$.  It
  follows immediately that $\bt_0(W,\partial W)=0$.  By duality,
  $\bt_4(W)=\bt_0(W,\partial W)=0$.

  (4) First we show that $\bt_0(M,\partial M)=0$; if $\partial M\ne
  \emptyset$, then $\bt_0(M,\partial M)=0$ as in~(3).  If the image of
  $\pi_1(M)$ in $G$, say $H$, is infinite, then $\ldim H_0(M;\N G) =
  \ldim \N G \otimes_{\C G} \C[G/H]=0$ by
  \cite[Lemma~6.33]{Lueck:2002-1}, \cite[Lemma~3.4]{Lueck:1998-1}.

  Now consider the long exact sequence of the triple $(W,\partial
  W,M')$ combined with an excision isomorphism:
  \[
  H_1(W,M';\N G) \to H_1(W,\partial W;\N G) \to
  H_0(\partial W,M';\N G) \cong H_0(M,\partial M;\N G)
  \]
  From (1) above with $M'$ in place of $M$, it follows that
  $\bt_1(W,\partial W)=0$.  By duality $\bt_3(W)=0$.

  (5) By the argument in (4), $|\Im\{\pi_1(W)\to G\}|=\infty$ implies
  $\bt_0(W)=0$.  By duality $\bt_4(W,\partial W)=0$.

  (6) The first equality follows from Poincar\'e duality.  By
  Lemma~\ref{lemma:H_i(W,M)-for-H_1-cobordism} (1),
  $b_3(W,M)=0=b_4(W,M)$.  From the long exact sequence of $(W,M)$, the
  second equality follows.

  (7) The conclusion follows from the exact sequence
  \[
  H_3(W,M)\to H_2(M) \to H_2(W) \to H_2(W,M) \to H_1(M) \to H_1(W)
  \]
  by observing that $H_3(W,M)\cong H^1(W,M')=0$ and $H_1(M)\cong
  H_1(W)$.  Similarly for $R$-coefficients.
\end{proof}

% Before start the proof, we observe that we may assume both $M$ and
% $M'$ in
% Theorem~\ref{theorem:amenable-signature-for-solvable-cobordism} have
% nonempty boundary, by the following argument.

% For closed $M$ and $M'$, then consider the 3-manifolds $M_0$ and
% $M'_0$ obtained by removing an open 3-ball.  If $W$ is an $h$-solvable
% cobordism between $M$ and $M'$, then the exterior of a properly
% embedded arc in $W$ with endpoints in $M$ and $M'$ is an $h$-solvable
% cobordism between $M_0$ and~$M'_0$, say~$W_0$.  Note that the
% fundamental groups of $M_0$, $M'_0$, $W_0$ are identical with those of
% $M$, $M'$, $W$.  Also, the Betti number conditions in
% Theorem~\ref{theorem:amenable-signature-for-solvable-cobordism} are
% maintained; standard Mayer-Vietoris argument shows
% $\bt_1(-M_0\cup_\partial M'_0)=\bt_1(M)+\bt_1(M_0)$,
% $b_1(M;R)=b_1(M_0;R)$ and $b_2(M;R)+b_3(M;R)=b_2(M_0;R)+b_3(M_0;R)$.
% Therefore,
% Theorem~\ref{theorem:amenable-signature-for-solvable-cobordism}
% applied to $(M_0, M'_0, W_0)$ gives $\rhot(-M_0\cup_\partial
% M'_0,\phi)=0$.  Since $-M_0\cup_\partial M'_0$ is merely the connected
% sum of $-M$ and $M'$ and $\rhot$-invariants are additive under
% connected sum, we have $\rhot(-M\cup M',\phi)=0$.

% Now using $V=(-M\cup M')\times[0,1]$ with a 1-handle
% attached as a cobordism between $-M\cup M'$ and $-M_0\cup_\partial
% M'_0$, one sees that $\rhot(-M\cup M',\phi)-\rhot(-M_0\cup_\partial
% M'_0,\phi)$ is the $L^2$-signature defect $\lsign(V)-\sign(V)$.  Since
% $V$ has no 2-handles, $\lsign(V)-\sign(V)=0$.

\begin{proof}[Proof of
  Theorem~\ref{theorem:amenable-signature-for-solvable-cobordism}]
  Recall from our assumption that $W$ is an $H_1$-cobordism between
  $M$ and $M'$ and $\phi\colon \pi_1(W)\to G$ is a homomorphism where
  $G$ is amenable and in $D(R)$ and $G^{(n+1)}=\{e\}$.

  Since $\partial W=M \cup_\partial M'$ over $G$, the
  $\rhot$-invariant is computed by the formula
  \[
  \rhot(M\cup_\partial M',\phi) = \lsign W - \sign W
  \]
  where $\lsign W$ denotes the $L^2$-signature of $W$ over~$\N G$, and
  $\sign W$ is the ordinary signature.

  Since $H_2(W)\to H_2(W,M)$ is surjective by
  Lemma~\ref{lemma:H_i(W,M)-for-H_1-cobordism}, the ordinary
  intersection pairing of $W$ is defined on $H_2(W,M)$ as a
  nonsingular pairing.  Furthermore, since there is a $0$-lagrangian,
  this intersection pairing is of the form
  $\left[\begin{smallmatrix}0&I\\I&*\end{smallmatrix}\right]$.  From
  this it follows that $\sign W=0$.

  In the remaining part of the proof we show $\lsign W=0$.  By
  definition $\lsign W$ is the $L^2$-signature of the intersection
  form
  \[
  H_2(W; \N G) \times H_2(W;\N G) \to \N G.
  \]
  This induces a hermitian form, say $\lambda_A$, on $A := \Im \{
  H_2(W;\N G) \to H_2(W,\partial W; \N G)\}$, and $\lambda_A$ is
  $L^2$-nonsingular in the sense of \cite[Section~3.1]{Cha:2010-01},
  namely both the kernel and cokernel of the associated homomorphism
  $A \to A^*=\Hom(A,\N G)$ given by $a\mapsto (b\mapsto \lambda(b,a))$
  have $L^2$-dimension zero.  We have $\lsign W = \lsign \lambda_A$
  since the intersection form vanishes on the image of $H_2(\partial
  W;\N G)$.

  Now we consider the three given cases.  To simplify notations we
  write $\pi=\pi_1(W)$, $m=\frac12 b_2(W,M)$.

  \smallskip

  \textbf{Case (I).}  For notational convenience, we exchange the
  r\^oles of $M$ and $M'$, namely assume $b_1^{(2)}(M')=0$.  Suppose
  $L$ is an $(n.5)$-lagrangian in $H_2(W;\Z[\pi/\pi^{(n+1)}])$.  Since
  $G^{(n+1)}$ is trivial, $\phi\colon \pi\to G$ induces a homomorphism
  $\pi/\pi^{(n+1)} \to G$.  We denote by $L'$ and $L''$ the images of
  $L$ in $H_2(W,M;\N G)$ and $H_2(W,\partial W;\N G)$ respectively.

  Note that the images of $L$ in $H_2(W,M;R)$ and $H_2(W,\partial W;R)$
  have $R$-dimension $m$, since the image of $L$ in $H_2(W;R)$ has
  0-duals by Lemma~\ref{lemma:langrangian-always-has-duals}~(1).  Applying
  Theorem~\ref{theorem:L2-dim-estimate-of-NG-homology} (2) to a
  collection of 2-cycles in $C_*(W,M;\Z G)$ generating the submodule
  $L' \subset H_2(W,M;\N G)$, and then by applying
  Lemma~\ref{lemma:L2-betti-numbers-of-H_1-cobordism}~(2), we obtain
  \[
  \ldim L' \ge \bt_2(W,M)-b_2(W,M;R)+m = m.
  \]

  By duality we have $\bt_2(M',\partial M') = \bt_1(M') = 0$.  Looking
  at the exact sequence of the triple $(W,\partial W, M)$
  \[
  H_2(M',\partial M';\N G) \to H_2(W,M;\N G) \xrightarrow{\alpha}
  H_2(W,\partial W;\N G),
  \]
  the second homomorphism $\alpha$ is $L^2$-monic, namely its kernel
  is of $L^2$-dimension zero.  From this and the above paragraph, it
  follows that $\ldim \alpha(L')=\ldim L' \ge m$, that is, $\ldim
  L''\ge m$.  (Recall that $\alpha(L')=L''$ by definition.)  On the
  other hand, since the map $H_2(W;\N G) \to H_2(W,\partial W;\N G)$
  factors through $H_2(W,M;\N G)$, we have $\ldim A \le \bt_2(W,M) =
  2m$ by Lemma~\ref{lemma:L2-betti-numbers-of-H_1-cobordism}~(2).

  Summarizing, $\lambda_A$ vanishes on the submodule $L''$ satisfying
  $\ldim L'' \ge \frac12 \ldim A$.  Now we apply the $L^2$-version of
  ``topologist's signature vanishing criterion'': if $\lambda\colon
  A\times A\to \N G$ is an $L^2$-nonsingular hermitian form over $\N
  G$ and there is a submodule $H \subset A$ such that $\lambda(H\times
  H)=0$ and $\ldim H \ge \frac 12 \ldim A$, then $\lsign \lambda=0$.
  (See \cite[Proposition~3.7]{Cha:2010-01}.)  In our case, it follows
  that $\lsign\lambda_A=0$.  This completes the proof of~(I).

  \smallskip

  \textbf{Case (II).}  Recall the assumption that $W$ is an
  $n.5$-solvable cobordism, $|\phi(\pi_1(M))|=\infty$, and
  $\bt_1(M\cup_\partial M') \ge b_1(M;R)+b_2(M;R)+b_3(M;R)-1$.  Let
  $A$, $L''$ be as in case~(I).  We will use alternative estimates of
  the $L^2$-dimensions to show that $\ldim L'' \ge \frac12 \ldim A$.
  First, applying Theorem~\ref{theorem:L2-dim-estimate-of-NG-homology}
  (2) to (the 2-cycles generating) $L''$ as a submodule of
  $H_2(W,\partial W;\N G)$ and then using Poincar\'e duality, we
  obtain
  \begin{align*}
    \ldim L'' & \ge \bt_2(W,\partial W) - b_2(W,\partial W;R)+m
    \\
    &= \bt_2(W)-b_2(W;R)+m.
  \end{align*}
  By looking at the homology long exact sequence for $(W,\partial W)$,
  we have
  \begin{align*}
    \ldim A &= \bt_2(W,\partial W)-\bt_1(\partial W)+\bt_1(W)
    \\
    &= \bt_2(W)-\bt_1(\partial W)+\bt_1(W)
  \end{align*}
  since $\bt_1(W,\partial W)=0$ by
  Lemma~\ref{lemma:L2-betti-numbers-of-H_1-cobordism}~(4).  It follows
  that
  \begin{align*}
    2\ldim L'' - \ldim A &\ge
    \bt_2(W)-\bt_1(W)-2b_2(W;R)+\bt_1(\partial W)+2m.
  \end{align*}
  Computing the Euler characteristic of $W$ using $b_i(W;R)$ and then
  using $\bt_i(W)$, we obtain
  \[
  \bt_2(W)-\bt_1(W) = b_2(W;R)-b_1(W;R)+1-b_3(M;R)
  \]
  by Lemma~\ref{lemma:L2-betti-numbers-of-H_1-cobordism} (3), (4),
  (5), and~(6).  Plugging this into the last inequality and then using
  the fact $H_1(W;R)\cong H_1(M;R)$ and
  Lemma~\ref{lemma:L2-betti-numbers-of-H_1-cobordism}~(7), it follows
  that
  \begin{align*}
    2\ldim L''-\ldim A &\ge \bt_1(\partial
    W)-b_1(W;R)-b_2(W;R)-b_3(M;R)+2m+1
    \\
    &= \bt_1(\partial W) - b_1(M;R) -b_2(M;R)-b_3(M;R)+1.
  \end{align*}
  Therefore $\ldim L'' \ge \frac12 \ldim A$ under our hypothesis.
  This proves~(II).

  \smallskip

  \textbf{Case (III).}  Now suppose $W$ is an $(n+1)$-solvable
  cobordism.  Suppose that $L$ is an $(n+1)$-lagrangian generated by
  $\ell_1,\ldots,\ell_m\in H_2(W;\Z[\pi/\pi^{(n+1)}])$ and
  $d_1,\ldots,d_m\in H_2(W;\Z[\pi/\pi^{(n+1)}])$ are $(n+1)$-duals
  satisfying $\lambda_{n+1}(\ell_i,d_j)=\delta_{ij}$.  Let $\ell''_i$
  be the image of $\ell_i$ in $H_2(W,\partial W;\N G)$.  The images
  $d_j''\in H_2(W; \N G)$ of the $d_j$ are dual to $\ell''_i$ with
  respect to the intersection pairing $H_2(W,\partial W; \N G)\times
  H_2(W; \N G) \to \N G$.  It follows that the $\ell''_i$ are linearly
  independent in $H_2(W,\partial W;\N G)$ over~$\N G$.  Therefore, the
  $\ell''_i$ generate a free $\N G$-module $L'' \subset A \subset
  H_2(W,\partial W;\N G)$ of rank $m$, and in particular $\ldim
  L''=m$.  Since $\ldim A \le \bt_2(W,M)=2m$ as in case (I), it
  follows that $\ldim L'' \ge \frac12 \ldim A$.  This completes the
  proof of~(III).
\end{proof}

\subsection{Vanishing of the first $L^2$-Betti number}
\label{subsection:vanishing-of-b^2_1}

In this subsection we discuss some cases to which Amenable Signature
Theorem~\ref{theorem:amenable-signature-for-solvable-cobordism}~(I)
applies.  We begin with a general statement providing several examples
with vanishing first $L^2$-Betti number, which generalizes
\cite[Lemma~3.12]{Cha:2010-01}, \cite[Proposition~2.11]{Cochran-Orr-Teichner:1999-1}.

\begin{proposition}
  \label{proposition:complexes-with-b^2_1=0}
  Suppose $G$ is amenable and lies in $D(R)$ for $R=\Z/p$ or $\Q$.
  Suppose $A \to X$ is a map between connected finite complexes $A$
  and $X$ inducing a surjection $H_1(A;R)\to H_1(X;R)$.  If
  $\phi\colon\pi_1(X)\to G$ is a homomorphism which induces an
  injection $\pi_1(A) \to G$, then $\bt_1(X;\N G)=\bt_1(A;\N G)=0$.
\end{proposition}

% For the proof of Proposition~\ref{proposition:complexes-with-b^2_1=0},
% we need the following ingredient to estimate the
% $L^2$-dimension. (Though we need only (1) here, we also state (2) for
% later use in this paper.)

% \begin{theorem}[Theorem~3.11 in \cite{Cha:2010-01}]
%   \label{theorem:L2-dim-estimate-of-NG-homology}
%   \leavevmode\Nopagebreak
%   \begin{enumerate}
%   \item Suppose $G$ is amenable and in $D(R)$ with $R=\Q$ or $\Z_p$,
%     and $C_*$ is a projective chain complex over $\Z G$ with $C_n$
%     finitely generated.  Then we have
%     \[
%     \ldim H_n(\N G \otimesover{\Z G}C_*) \le \dim_R
%     H_n(R\otimesover{\Z G}C_*).
%     \]
%   \item In addition, if $\{x_i\}_{i\in I}$ is a collection of
%     $n$-cycles in $C_n$, then for the submodules $H\subset H_n(\N G
%     \otimes C_*)$ and $\overline H \subset H_n(R\otimes C_*)$
%     generated by $\{[1_{\N G}\otimes x_i]\}_{i\in I}$ and
%     $\{[1_{R} \otimes x_i]\}_{i\in I}$, respectively, we have
%     \[
%     \ldim H_n(\N G \otimesover{\Z G} C_*)-\ldim H \le \dim_R
%     H_n(R\otimesover{\Z G} C_*) - \dim_R \overline{H}.
%     \]
%   \end{enumerate}
% \end{theorem}

\begin{proof}
  %[Proof of Proposition~\ref{proposition:complexes-with-b^2_1=0}] 
  By the assumption, $H_1(X,A;R)=0$.  By applying
  Theorem~\ref{theorem:L2-dim-estimate-of-NG-homology}~(1) to the
  chain complex $C_*(X,A;\Z G)$, we obtain $\bt_1(X,A)=0$.  From the
  $\N G$-coefficient homology long exact sequence for $(X,A)$, it
  follows that $\bt_1(X) \le \bt_1(A)$.  Since the induced map
  $\pi_1(A)\to G$ is injective, the $G$-cover of $A$ is a disjoint
  union of copies of the universal cover of~$A$.  Consequently
  $H_1(A;\C G)=0$.  By the universal coefficient spectral sequence,
  $H_1(A;\N G) = \Tor_1^{\C G} (\N G, H_0(A;\C G))$.  Since $G$ is
  amenable, $\bt_1(A) = \ldim \Tor_1^{\C G} (\N G, H_0(A;\C G))= 0$ by
  \cite[Theorem~6.37]{Lueck:2002-1}.
\end{proof}

\subsubsection*{Exteriors of two-component links with nonvanishing
  linking number}

\begin{theorem}
  \label{theorem:two-component-links-with-b^2_1=0}
  Suppose $L$ is a two component link with exterior $M$, and suppose
  either \textup{(i)} $R=\Q$ and $\lk(L)\ne 0$, or \textup{(ii)}
  $R=\Z/p$ and $\lk(L)$ is relatively prime to~$p$.  Suppose $G$ is an
  amenable group in $D(R)$.  If $\phi\colon \pi_1(M) \to G$ is a
  homomorphism which the abelianization $\pi_1(M)\to \Z^2$ factors
  through, then $\bt_1(M;\N G)=0$.
\end{theorem}

\begin{proof}
  Let $A$ be a (toral) boundary component of~$M$.  From the linking
  number condition, it follows that $H_1(A;R) \to H_1(M;R)$ is
  an isomorphism.  Also, the composition
  \[
  \Z^2=\pi_1(A)\to \pi_1(M) \xrightarrow{\text{ab}} \Z^2
  \]
  is injective, since tensoring it with $R$ one obtains $H_1(A;R) \to
  H_1(M;R)$.  Therefore by
  Proposition~\ref{proposition:complexes-with-b^2_1=0} we conclude
  that $\bt_1(M)=0$.
\end{proof}

We will investigate an application of
Theorem~\ref{theorem:amenable-signature-for-solvable-cobordism}~(2) to
this case in Section~\ref{section:concordance-to-hopf-link}.

\subsubsection*{Knot exteriors}

Proposition~\ref{proposition:complexes-with-b^2_1=0} also applies to
$(X,A)=(M,\mu)$, where $M$ is the exterior (or the zero-surgery
manifold) of a knot and $\mu$ is a meridian.  Indeed this case is none
more than \cite[Lemma 3.12]{Cha:2010-01}, as done in \cite[Proof of
Theorem~3.2]{Cha:2010-01}.  In the special case of a PTFA group $G$, a
similar result appeared earlier in \cite[Proposition
2.11]{Cochran-Orr-Teichner:1999-1}.

\subsection{Relationship with and generalizations of previously known
  results}
\label{subsection:relationship-with-previous-results}

Here we discuss some known results on $L^2$-signature obstructions as
special cases of
Theorem~\ref{theorem:amenable-signature-for-solvable-cobordism}.

\subsubsection*{Obstructions to knots being $n.5$-solvable}

In \cite{Cochran-Orr-Teichner:1999-1}, the notion of an $h$-solvable
knot was first introduced.  A knot $K$ is defined to be $h$-solvable
if its zero-surgery bounds a 4-manifold $W$ called an $h$-solution
(see \cite[Definitions 1.2, 8.5, 8,7]{Cochran-Orr-Teichner:1999-1}),
which is easily seen to be a spin $h$-solvable cobordism between the
exterior of $K$ and that of a trivial knot.  The following theorem,
which appeared in \cite{Cha:2010-01}, is an immediate consequence of
our Amenable Signature
Theorem~\ref{theorem:amenable-signature-for-solvable-cobordism} (see
also the last paragraph of
Section~\ref{subsection:vanishing-of-b^2_1}).

\begin{theorem}[{\cite[Theorem~1.3]{Cha:2010-01}}]
  \label{theorem:knot-solvability-amenable-obstruction}
  If $K$ is an $n.5$-solvable knot, $R=\Q$ or $\Z/p$, $G$ is an
  amenable group in $D(R)$, $G^{(n+1)}=\{e\}$, and $\phi\colon
  \pi_1(M(K))\to G$ is a homomorphism that sends a meridian to an
  infinite order element and extends to an $n.5$-solution, then
  $\rhot(M(K),\phi)=0$.
\end{theorem}

We note that \cite[Theorem~3.2]{Cha:2010-01}, which is a slightly
stronger version of
Theorem~\ref{theorem:knot-solvability-amenable-obstruction}, is also a
consequence of
Theorem~\ref{theorem:amenable-signature-for-solvable-cobordism}.
Also, the following theorem of
Cochran-Orr-Teichner~\cite{Cochran-Orr-Teichner:1999-1} is a
consequence of our
Theorem~\ref{theorem:amenable-signature-for-solvable-cobordism} since
it follows from
Theorem~\ref{theorem:knot-solvability-amenable-obstruction} as pointed
out in~\cite{Cha:2010-01}:

\begin{theorem}[{\cite[Theorem~4.2]{Cochran-Orr-Teichner:1999-1}}]
  \label{theorem:knot-solvability-ptfa-obstruction}
  If $K$ is an $n.5$-solvable knot, $G$ is poly-torsion-free-abelian,
  $G^{(n+1)}=\{e\}$, and $\phi\colon \pi_1(M(K))\to G$ is a nontrivial
  homomorphism extending to an $n.5$-solution, then
  $\rhot(M(K),\phi)=0$.
\end{theorem}

\begin{remark}
  On the other hand, the homology cobordism result and concordance
  result given in \cite[Theorem 7.1]{Cha-Orr:2009-01} and
  \cite[Theorem~1.2]{Cha:2010-01} are potentially stronger than our
  Amenable Signature
  Theorem~\ref{theorem:amenable-signature-for-solvable-cobordism}; in
  particular these do \emph{not} require that the group $G$ is
  solvable.  It is an extremely interesting open question if certain
  non-solvable amenable signatures actually reveal something beyond
  solvable groups.
\end{remark}

\subsubsection*{Harvey's $\rho_n$-invariant and Whitney tower cobordism}

In this subsection we observe that the homology cobordism invariants
of Harvey~\cite{Harvey:2006-1}
are indeed invariant under Whitney tower cobordism.  

For a group $G$, Harvey defined a series of normal subgroups
$G=G^{(0)}_H \supset G^{(1)}_H \supset \cdots \supset G^{(n)}_H
\supset \cdots$ which is called the \emph{torsion-free derived
  series}~\cite{Cochran-Harvey:2004-1,Harvey:2006-1}.  A key theorem
of Harvey \cite[Theorem~4.2]{Harvey:2006-1} says the following: for a
closed 3-manifold $M$, $\rho_n(M):=\rhot(M, \pi_1(M) \to
\pi_1(M)/\pi_1(M)^{(n+1)}_H)\in \R$ is a homology cobordism invariant.
This can be strengthened as follows:

% Using homology localizations of groups and modules together with
% amenable $L^2$-techniques, a newer invariant has been discovered by
% Orr and the author.  In \cite{Cha-Orr:2009-01}, they defined the
% $R$-coefficient \emph{Vogel-Cohn local derived series} $\{G^{(n)}\}$
% of a group~$G$:
% \[
% G=G^{(0)} \supset G^{(1)} \supset \cdots \supset G^{(n)} \supset
% \cdots
% \]
% As an abuse of notation, in this subsection $G^{(n)}$ designates
% temporarily this series appeared in \cite{Cha-Orr:2009-01}, not the
% standard derived series.  They showed the following:

% \begin{theorem}[\cite{Cha-Orr:2009-01}]
%   \label{theorem:cha-orr-rho-invariant}
%   For a closed 3-manifold $M$, let $\phi_n\colon \pi_1(M) \to
%   \pi_1(M)/\pi_1(M)^{(n)}$ be the projection.  Then $\rhot(M,\phi_n)$
%   is an $R$-homology cobordism invariant.
% \end{theorem}

% While \cite{Cha-Orr:2009-01} gives more general results concerning
% homology with twisted coefficients, here we will consider the above
% untwisted homology cobordism case only.

% It is known that $\rhot(M,\phi_n)$ can give more information even when
% $\rho_n(M)$ vanish, essentially since $\pi/\pi^{(n)}$ is larger than
% (or equal to) $\pi/\pi^{(n)}_H$ (see \cite[Theorem
% 4.1]{Cha-Orr:2009-01}).  More interestingly the former has torsion in
% many cases, while the latter is torsion-free.

% Our amenable signature theorem enables us to show the invariance of
% the two invariants under Whitney tower cobordism:

\begin{theorem}
  \label{theorem:whitney-tower-and-harvey-cha-orr-rho-invariant}
  Suppose $M$ and $M'$ are closed 3-manifolds.  Let $\bt_i(M)$ be the
  $L^2$-Betti number over $\N(\pi_1(M)/\pi_1(M)^{(n+1)}_H)$, and
  define $\bt_i(M')$ similarly.
  \begin{enumerate}
  \item If $M$ and $M'$ are height $n+1$ Whitney tower cobordant, then
    $\rho_n(M)=\rho_n(M')$.
    %and $\rhot(M,\phi_n) = \rhot(M',\phi_n)$.
  \item If $M$ and $M'$ are height $n.5$ Whitney tower cobordant and
    either $\bt_1(M)=0$ or
    \[
    b_1(M)\ne 0 \,\text{ and }\, \bt_1(M)+\bt_1(M') \ge b_1(M)+b_2(M)+b_3(M)-1,
    \]
    then $\rho_n(M) = \rho_n(M')$.  
    % Similarly for $\rhot(-;\phi_n)$,
    % replacing $(-)^{(n)}_H$ by $(-)^{(n)}$.
  \end{enumerate}
\end{theorem}

\begin{proof}
  We will prove (1) and (2) simultaneously.  By
  Theorem~\ref{theorem:whitney-tower-and-solvability} and
  Remark~\ref{remark:whitney-tower-and-solvability-n-case}, there is
  an $h$-solvable cobordism $W$ between $M$ and~$M'$, where $h=n+1$ in
  case of (1) and $h=n.5$ in case of~(2).  We have $H_1(M)\cong
  H_1(W)$.  Also $H_2(W;\Z[\pi_1(W)/\pi_1(W)^{(n)}])\rightarrow
  H_2(W)$ is surjective, since there is an $n$-lagrangian admitting
  $n$-duals.  Since $\pi_1(W)^{(n)} \subset \pi_1(W)^{(n)}_H$, we can
  apply the Dwyer-type injectivity theorem
  \cite[Theorem~2.1]{Cochran-Harvey:2006-01} to $\pi_1(M)\to \pi_1(W)$
  to conclude that the quotient $\pi_1(M)/\pi_1(M)^{(n+1)}_H$ injects
  into $\Gamma:=\pi_1(W)/\pi_1(W)^{(n+1)}_H$ under the
  inclusion-induced map.  By the $L^2$-induction property (for
  example, see \cite[p.~8~(2.3)]{Cheeger-Gromov:1985-1},
  \cite[Proposition~5.13]{Cochran-Orr-Teichner:1999-1}), it follows
  that $\rho_n(M) = \rhot(M,\pi_1(M)\to \Gamma)$ and
  $\bt_i(M)=\bt_i(M,\N\Gamma)$.  Similarly for~$M'$.

  Our $\Gamma$ satisfies $\Gamma^{(n+1)}=\{e\}$, and is known to be
  amenable and in~$D(\Q)$.  Also, $b_1(M)\ne 0$ implies
  $|\Gamma|=\infty$.  Therefore by applying
  Theorem~\ref{theorem:amenable-signature-for-solvable-cobordism}
  (III) and (II), it follows that $\rhot(M,\pi_1(M)\to \Gamma) =
  \rhot(M',\pi_1(M')\to \Gamma)$ from the hypothesis of (1) and (2)
  respectively.
\end{proof}

\subsubsection*{Harvey's $\rho_n$-invariants and $h$-solvable links}

Harvey and Cochran-Harvey also gave $\rho_n$-invariant obstructions to
being $h$-solvable \cite[Theorem~6.4]{Harvey:2006-1},
\cite[Theorem~4.9, Corollary~4.10]{Cochran-Harvey:2006-01}.  Their
relationship with our Amenable Signature
Theorem~\ref{theorem:amenable-signature-for-solvable-cobordism} is
best illustrated in case of links, as discussed below.

The notion of an $h$-solution of a link $L$ in
\cite{Cochran-Orr-Teichner:1999-1} is related to a spin $h$-solvable
cobordism between the exterior $E_L$ of $L$ and the trivial link
exterior~$E_O$, similarly to the knot case, though the details are
slightly more technical.  We give an outline below, omitting details.
Let $N$ be the exterior of the standard slice disks in $D^4$ for a
trivial link $O$.  Given a spin $h$-solvable cobordism $W$ between
$E_L$ and $E_O$, it can be seen that $V:=W\cup_{E_O} N$ is an
$h$-solution for $L$ in the sense of
\cite{Cochran-Orr-Teichner:1999-1} by a straightforward computation of
$H_1$ and~$H_2$.  Conversely, if $V$ is an $h$-solution for~$L$, it
turns out that there is an embedding of $N$ into $V$ in such a way
that $W:=\overline{V-N}$ is an $h$-solvable cobordism between $E_L$
and~$E_O$.  In fact, if one chooses disjoint arcs $\gamma_i$ in $V$
joining a fixed basepoint in $\inte(V)$ to a meridian $\mu_i$ of the
$i$th component of $L$, then a regular neighborhood of $\bigcup_i
(\gamma_i \cup \mu_i)$ is homeomorphic to~$N$.

Now suppose $L$ has $m$ components and $\pi_1(W)\to G$ is given as in
Theorem~\ref{theorem:amenable-signature-for-solvable-cobordism}.
Then, it turns out that the Betti number condition in
Theorem~\ref{theorem:amenable-signature-for-solvable-cobordism} (II)
is equivalent to $\bt_1(M_L) \ge m-1$, if the image of each meridian
of $L$ in $G$ has infinite order.  So, for $G$ PTFA, one recovers the
Cochran-Harvey rank condition in \cite[Theorem~4.9,
Corollary~4.10]{Cochran-Harvey:2006-01}.  Indeed, in our Betti number
condition $\bt_1(E_L\cup_\partial E_O) \ge
b_1(E_O;R)+b_2(E_O;R)+b_3(E_O;R)-1$, one can show that $b_1(E_O)=m$,
$b_2(E_O)=m-1$, and furthermore $\bt_1(E_L\cup_\partial
E_O)=\bt_1(E_L)+\bt_1(E_O)$ and $\bt_1(E_L)=\bt_1(M_L)$,
$\bt_1(E_O)=m-1$.  From this it follows that our Amenable Signature
Theorem~\ref{theorem:amenable-signature-for-solvable-cobordism} (II)
specializes to the $\rho_n$-invariant obstruction to links being
$n.5$-solvable \cite[Corollary~4.10]{Cochran-Harvey:2006-01}.

% \section{Example: Whitney tower cobordism of 3-manifolds with local
%   hidden torsion}

% In this section, as an application of Amenable Signature
% Theorem~\ref{theorem:amenable-signature-for-solvable-cobordism}, we
% show that the examples of 3-manifolds in \cite{Cha-Orr:2011-01} are
% actually not height 2.5 Whitney tower cobordant.

\section{Grope and Whitney tower concordance to the Hopf link}
\label{section:concordance-to-hopf-link}

\def\Lh{H}

In this section we give an application to concordance of links with
nonvanishing linking number.  Our goal is to prove the following
result:

\begin{theorem}
  \label{theorem:whitney-tower-concordance-to-hopf-link}
  For any integer $n>2$, there are links with two unknotted components
  which are height $n$ grope concordant (and consequently height $n$
  Whitney tower concordant) to the Hopf link, but not height $n.5$
  Whitney tower concordant (and consequently not height $n.5$ grope
  concordant) to the Hopf link.
\end{theorem}

We remark that the blow-down technique, namely performing
$(\pm1)$-surgery along one component and then studying the concordance
of the resulting knot, might be useful in showing that our links in
Theorem~\ref{theorem:whitney-tower-concordance-to-hopf-link} are not
slice.  Nonetheless, it is unknown whether the blow-down technique
could yield any interesting conclusion about the height of Whitney
towers and gropes as our method does.

\subsection{Satellite construction and capped gropes}

To construct our example, we will use a standard satellite
construction (often called infection) described as follows: let $L$ be
a link in $S^3$, and $\eta$ be an unknotted circle in $S^3$ which is
disjoint from~$L$.  Remove a tubular neighborhood of $\eta$ from
$S^3$, and then attach the exterior of a knot $J$ along an orientation
reversing homeomorphism on the boundary that identifies the meridian
and 0-longitude of $J$ with the 0-longitude and meridian of~$\eta$,
respectively.  The resulting 3-manifold is again homeomorphic to
$S^3$, and the image of $L$ under this homeomorphism is a new link in
$S^3$, which we denote by~$L(\eta,J)$.  We note that the same
construction applied to a framed circle $\eta$ embedded in a
3-manifold $M$ gives a new 3-manifold, which we denote by $M(\eta,J)$.

Recall that a \emph{capped grope} is defined to be a grope with
2-disks attached along each standard symplectic basis curves of the
top layer surfaces (see, e.g.,
\cite[Chapter~2]{Freedman-Quinn:1990-1}).  These additional 2-disks
are called the \emph{caps}, and the remaining grope part is called the
\emph{body}.  We remark that an embedded capped grope in this paper
designates a capped grope embedded in a 4-manifold.  In particular not
only the body but also all caps are embedded, while capped gropes with
immersed caps are often used in the literature.

Our construction of grope concordance depends on the following
observation.  For convenience, we use the following terms.  Recall
that we denote the exterior of a link $L$ by~$E_L$.

\begin{definition}
  We call $(L,\eta)$ a \emph{satellite configuration of height $n$} if
  $L$ is a link in $S^3$, $\eta$ is an unknotted circle in $S^3$
  disjoint to $L$, and the 0-linking parallel of $\eta$ in
  $E_\eta=E_\eta\times 0$ bounds a height $n$ capped grope $G$
  embedded in $E_\eta\times [0,1]$ with body disjoint to
  $L\times[0,1]$.  We call $G$ a \emph{satellite capped grope
    for~$(L,\eta)$}.
\end{definition}

We remark that by definition a satellite configuration $(L,\eta)$ of
height zero is merely a link $L$ with an unknotted curve $\eta$
disjoint to~$L$.

% For example, our $(L_0, \eta)$ given in Figure~\ref{figure:seed-link}
% is a satellite configuration of height one by
% Lemma~\ref{lemma:properties-of-seed-link}.

\begin{proposition}[Composition of satellite configurations]
  \label{proposition:composition-satellite-conf}
  Suppose $(L,\eta)$ is a satellite configuration of height $n$, and
  $(K,\alpha)$ is a satellite configuration of height $m>0$ with $K$ a
  knot.  Let $L'=L(\eta,K)$.  Then, viewing $\alpha$ as a curve in
  $E_{L'}$ via $\alpha\subset E_K \subset E_K\cup E_{L\cup\eta} =
  E_{L'}$, $(L',\alpha)$ is a satellite configuration of height $n+m$.
\end{proposition}

\begin{proof}
  Suppose $H$ is a satellite capped grope of height $m$ for
  $(K,\alpha)$.  We may assume that the intersection of $H$ with a
  tubular neighborhood of $K\times[0,1]$ consists of disjoint disks
  $D_1,\ldots,D_r$ lying on caps of~$H$ and $\partial D_i$ is of the
  form $\mu\times t_i$, where $\mu$ is a fixed meridian of $K$ and
  $t_i\in(0,1)$ are distinct points.

  Suppose $G$ is a satellite capped grope for $(L,\eta)$.  Let $U$ be
  the union of $r$ parallel copies of $G$ in $E_\eta\times[0,1]$ such
  that the boundary of $U$ is $\bigcup_i (\text{parallel copy of
  }\eta)\times t_i$.  Let $V=H \cap (E_K\times [0,1])$.  Note that the
  boundary of $V$ is $\bigcup_i \mu\times t_i$ and $\mu$ is identified
  with a parallel copy of $\eta$ under the satellite construction.
  Now
  \[
  \def\cupo#1{\cup}
  U \cupo{\eta\times t_i = \mu\times t_i} V \subset \big(E_\eta
  \times[0,1]\big) \cupo{\partial E_\eta\times[0,1] = \partial E_K
    \times[0,1]} \big (E_K\times[0,1]\big) \cong S^3 \times[0,1]
  \]
  is a desired satellite capped grope of height $n+m$ for
  $(L',\alpha)$.
\end{proof}

We remark that the construction used above may be compared to
\cite[Section~3]{Horn:2010-1}.

\subsection{Building blocks}

\subsubsection*{A seed link}

We start with a two-component link $L_0$ given in
Figure~\ref{figure:seed-link}.  (The curve $\eta$ and the dotted arc
$\gamma$ are not parts of~$L_0$.)

\begin{figure}[h]
  \labellist
  \small\hair 0mm
%  \pinlabel {$\gamma$} [r] at 34 150
  \pinlabel {$\gamma$} [r] at 230 105
  \pinlabel {$\eta$} [r] at 100 15
  \pinlabel {$+a$ full twists} [l] at 158 5
  \pinlabel {$-a$ full twists} [l] at 158 205
  \endlabellist
  \includegraphics[scale=.6]{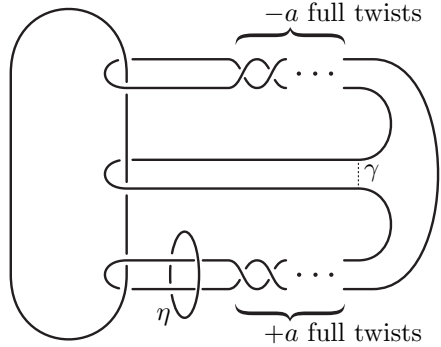}
  \caption{A link which is concordant to the Hopf link}
  \label{figure:seed-link}
\end{figure}

The following properties of $L_0$ and $\eta$ in
Figure~\ref{figure:seed-link} will be crucially used in this section.
In fact any $(L_0,\eta)$ with these properties can be used in place of
our $(L_0,\eta)$.

\begin{lemma}
  \label{lemma:properties-of-seed-link}
  \leavevmode\Nopagebreak
  \begin{enumerate}
  \item The link $L_0$ is concordant to the Hopf link.
  \item $(L_0,\eta)$ is a satellite configuration of height one.
  \item For $a\ne 0$, the Alexander module
    $A=H_1(E_{L_0};\Z[x^{\pm1},y^{\pm1}])$ of~$L_0$ is a nonzero
    $\Z[x^{\pm1},y^{\pm1}]$-torsion module generated by the homology
    class of~$\eta$.
  \end{enumerate}
\end{lemma}

In (3) above, the variables $x$ and $y$ correspond to the right and
left components in Figure~\ref{figure:seed-link}, respectively.

\begin{proof}
  (1) Applying a saddle move along the dotted arc $\gamma$ (see
  Figure~\ref{figure:seed-link}), the right component splits into two
  new components.  One of these (which is the ``broken middle tine'')
  forms a Hopf link together with the left component of~$L_0$.  The
  other new component is easily isotoped to a separated unknotted
  circle since the $\pm a$ twistings are now eliminated.  This gives a
  concordance in $S^3\times[0,1]$ consisting of two annuli, one of
  which is a straight product of the left component of $L_0$ and
  $[0,1]$, and another annulus has one saddle point and one local
  maximum.

  (2) By tubing the obvious 2-disk, it is easily seen that $\eta$
  bounds an embedded genus one surface in the 3-space which is
  disjoint to~$L_0$.  In addition one can attach two caps which meet
  the left and right components of $L_0$ once, respectively.  This
  gives a desired satellite capped grope of height one.

  (3) A straightforward homology computation shows that $L_0$ has
  Alexander module
  \[
  H_1(E_{L_0};\Z[x^{\pm1},y^{\pm1}])= \Z[x^{\pm1},y^{\pm1}]/\langle
  f\overline f\rangle
  \]
  generated by $[\eta]$, where $f=a(x+y^{-1}-xy^{-1}-1)+1$.  Details
  are as follows.

  The link $L_0$ can be represented as the leftmost diagram in
  Figure~\ref{figure:seed-link-as-surgery}, where $(\pm 1/a)$-surgery
  curves are used instead of the $\pm a$ twists.  By isotopy, we
  obtain the rightmost surgery diagram in
  Figure~\ref{figure:seed-link-as-surgery} with $L_0$ given as the
  standard Hopf link in~$S^3$.  By taking the universal abelian cover
  of the exterior $S^1\times S^1\times[0,1]$ of this Hopf link and
  then taking the lifts of the $(\pm 1/a)$-surgery curves, we obtain
  the universal abelian cover of~$E_{L_0}$ as a surgery diagram in
  $\R^2\times[0,1]$, which is shown in
  Figure~\ref{figure:seed-link-cover-3d}.  It is isotopic to the
  ``flattened'' diagram in Figure~\ref{figure:seed-link-cover}, where
  the covering transformations $x$ and $y$ are given by shifting right
  and down, respectively.  Obviously the framing on the lifts of the
  $(1/a)$-surgery curve is again~$1/a$.  For the $(-1/a)$-surgery
  curve, since the $+1$ twists on the horizontal bands in
  Figure~\ref{figure:seed-link-cover} contribute additional $-2$ to
  the writhe of the base curve, if the framing on a lift is $p/q$ then
  the base curve framing must be $(p-2q)/q$.  It follows that
  $p/q=(2a-1)/a$ as in Figure~\ref{figure:seed-link-cover}.

  \begin{figure}[h]
    \labellist
    \small\hair 0mm
    \pinlabel {$1/a$} at 55 110
    \pinlabel {$-1/a$} at 55 5
    \pinlabel {$1/a$} at 185 110
    \pinlabel {$-1/a$} at 185 23
    \pinlabel {$1/a$} at 357 43
    \pinlabel {$-1/a$} at 278 18
    \pinlabel {\large$\approx$} at 113 57
    \pinlabel {\large$\approx$} at 250 57
    \endlabellist
    \includegraphics{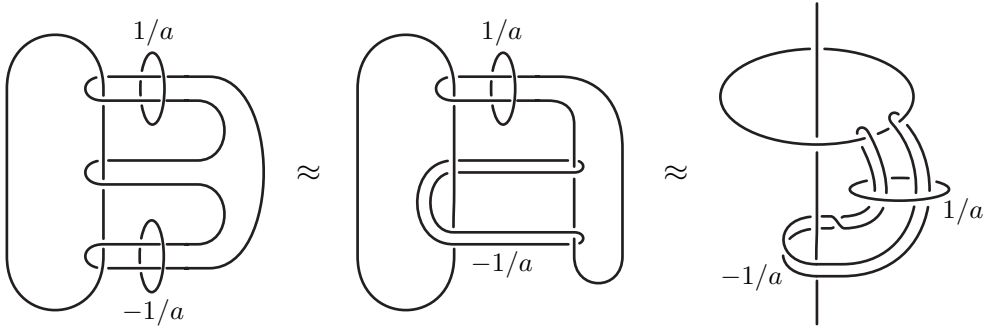}
    \caption{A surgery presentation of the seed link.}
    \label{figure:seed-link-as-surgery}
  \end{figure}

  \begin{figure}[h]
    \includegraphics{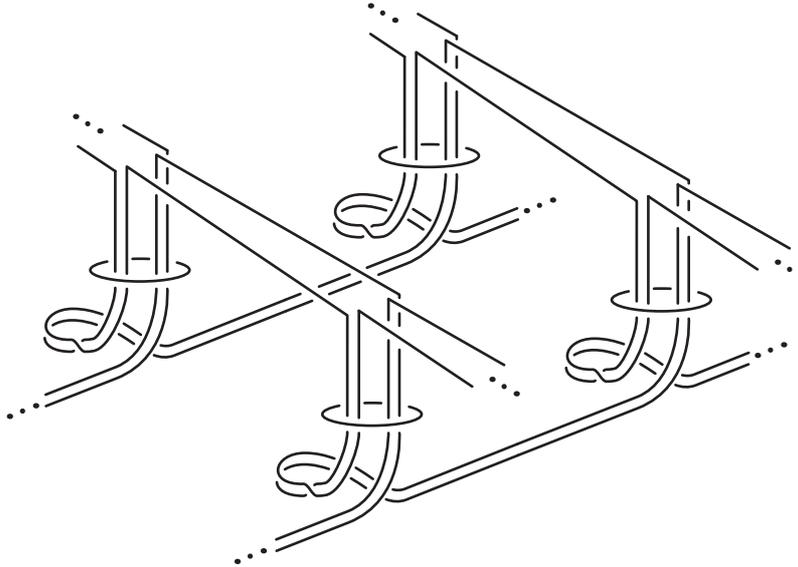}
    \caption{The universal abelian cover of the seed link exterior.}
    \label{figure:seed-link-cover-3d}
  \end{figure}
  
  \begin{figure}[h]
    \labellist
    \footnotesize\hair 0mm
    \pinlabel {$1/a$} at 45 145
    \pinlabel {$1/a$} at 139 145
    \pinlabel {$1/a$} at 45 63
    \pinlabel {$1/a$} at 139 63
    \pinlabel {$(2a-1)/a$} [r] at 44 115
    \pinlabel {$(2a-1)/a$} [r] at 138 115
    \pinlabel {$(2a-1)/a$} [r] at 44 33
    \pinlabel {$(2a-1)/a$} [r] at 138 33
    \pinlabel {$u$} at 181 148
    \pinlabel {$v$} at 195 161
    \endlabellist
    \includegraphics{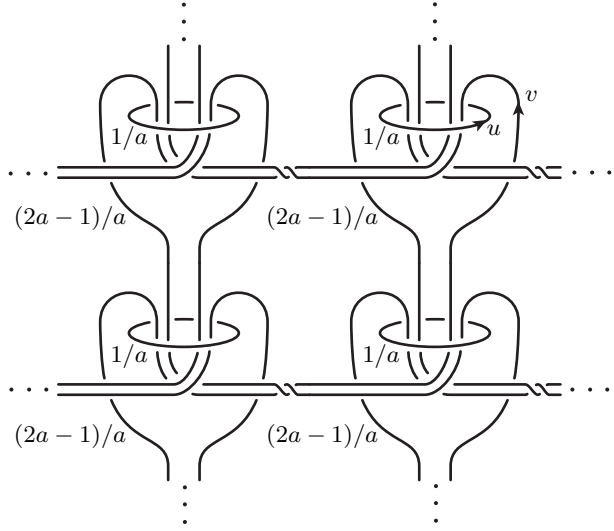}
    \caption{Another diagram of the universal abelian cover.}
    \label{figure:seed-link-cover}
  \end{figure}

  The first homology of the universal abelian cover of $E_{L_0}$ has
  two generators, namely the meridians of surgery curves $u$ and $v$
  in Figure~\ref{figure:seed-link-cover}, as a module over $\Z[x^{\pm
    1}, y^{\pm 1}]$.  The defining relations given by the surgery are
  read from the linking numbers of the various translations
  $x^iy^j(u)$, $x^iy^j(v)$ with the $(1/a)$-curve and
  $((2a-1)/a)$-curve on the boundary of the tubular neighborhood of
  $u$ and~$v$.  From this we obtain a presentation matrix
  \[
  \begin{bmatrix}
    -1 & a(x+y^{-1}-xy^{-1}-1) \\
    a(x^{-1}+y-x^{-1}y-1)  & a(x+x^{-1}+y+y^{-1}-xy^{-1}-x^{-1}y-2)+1 
  \end{bmatrix}
  \]
  of $H_1(E_{L_0};\Z[x^{\pm 1}, y^{\pm 1}])$, with respect to the
  meridians $\mu_u$ and $\mu_v$ of the curves $u$ and~$v$.  Here the
  first and second rows corresponds the $(1/a)$-curve and
  $((2a-1)/a)$-curve of $u$ and $v$ respectively, and the first and
  second columns corresponds to $u$ and~$v$.  For example, the
  $(1,2)$-entry is equal to $\sum_{i,j} \lk\big(x^iy^j(v),
  (\text{$(1/a)$-curve around $u$})\big) x^iy^j$.

  It follows that $H_1(E_{L_0};\Z[x^{\pm 1}, y^{\pm
    1}])\cong \Z[x^{\pm 1}, y^{\pm 1}]/\langle f\overline f\rangle$,
  generated by $\mu_v$.  Since the curve $\eta$ in
  Figure~\ref{figure:seed-link} is isotopic to (the zero-linking
  longitude of) the projection of $v$, $a\cdot[\eta]$ is equal to
  $\mu_v$ in $H_1(E_{L_0};\Z[x^{\pm 1}, y^{\pm 1}])$.  It follows that
  $[\eta]$ is a generator.
\end{proof}

% \begin{lemma}
%   \label{lemma:alexander-polynomial-of-L0}
%   The link $L_0$ has Alexander module
%   \[
%   H_1(E_L;\Z[x^{\pm1},y^{\pm1}])= \Z[x^{\pm1},y^{\pm1}]/\langle
%   f\overline f\rangle
%   \]
%   where $f=n(-x-y^{-1}+xy^{-1}+1)+1$.  Also, the homology class of the
%   curve $\eta$ illustrated in Figure~\ref{figure:seed-link} is a
%   generator of $H_1(E_L;\Z[x^{\pm1},y^{\pm1}])$.
% \end{lemma}

% A straightforward homology computation proves
% Lemma~\ref{lemma:alexander-polynomial-of-L0}.  We omit details.

\subsubsection*{Seed knots}

Another building block is a knot satellite configuration as in
Lemma~\ref{lemma:seed-knot} stated below.

\begin{lemma}
  \label{lemma:seed-knot}
  There exist satellite configurations $(K,\alpha)$ of height one with
  $K$ a slice knot such that the Alexander module
  $H_1(E_K;\Z[t^{\pm1}])$ of $K$ is nonzero and generated by the
  homology class of~$\alpha$.
\end{lemma}

It is folklore that such $(K,\alpha)$ is not rare.  For example,
\cite[Theorem~5.18]{Cha:2003-1} gives a construction of a ribbon knot
with Alexander module $\Z[t^{\pm1}]/\langle P(t)^2\rangle$ for any
polynomial $P(t)$ with integral coefficients satisfying $P(1)=\pm1$
and $P(t^{-1})=P(t)$ up to multiplication by~$\pm t^r$, and for this
knot it is not difficult to see that there is a curve $\alpha$ with
the desired property.  As an explicit example, straightforward
computation shows that Stevedore's knot $K$ with the curve $\alpha$
illustrated in Figure~\ref{figure:stevedore-knot} satisfies
Lemma~\ref{lemma:seed-knot}.

\begin{figure}[h]
  \labellist
  \small\hair 0mm
  \pinlabel {$\alpha$} [r] at 94 40
  \endlabellist
  \includegraphics[scale=.8]{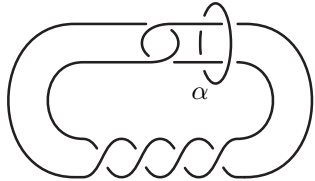}
  \caption{Stevedore's knot with a curve $\alpha$ bounding a grope of
    height one.}
  \label{figure:stevedore-knot}
\end{figure}

\subsubsection*{Cochran-Teichner knot}

Let $J$ be the knot given by Cochran and Teichner
in~\cite[Figure~3.6]{Cochran-Teichner:2003-1}.  
% We illustrate $J$ in
% Figure~\ref{figure:cochran-teichner-knot} for the convenience of the
% readers.
% \begin{figure}[H]
%   \fbox{ct-knot.eps}
%   \caption{Cochran-Teichner's knot.}
%   \label{figure:cochran-teichner-knot}
% \end{figure}
We need the following nice properties of~$J$.

\begin{lemma}[\cite{Cochran-Teichner:2003-1}]
  \label{lemma:cochran-teichner-knot}
  \leavevmode\Nopagebreak
  \begin{enumerate}
  \item $\int_{S^1} \sigma_J(\omega)\, d\omega \ne 0$, where
    $\sigma_J$ is the Levine-Tristram signature function of~$J$.
  \item For any connected sum $J_0$ of copies of $J$, if $(L,\eta)$ is
    a satellite configuration of height $n$, then $L(\eta,J_0)$ is
    height $n+2$ grope concordant to~$L$.
  \end{enumerate}
\end{lemma}

We note that by applying (2) to
$(L,\eta)=(\text{unknot},\text{meridian})$, $J=L(\eta,J)$ bounds a
height 2 grope in~$D^4$.

\begin{proof}
  (1) is \cite[Lemma~4.5]{Cochran-Teichner:2003-1}.  (2) is an
  immediate consequence of (a link version of)
  \cite[Corollary~3.14]{Cochran-Teichner:2003-1}.
\end{proof}

\subsection{Construction of examples}
\label{subsection:construction-of-examples}

In the remaining part of this section we assume the following:

\begin{enumerate}
\item[(C1)] $(L_0,\eta)$ is a satellite configuration satisfying
  Lemma~\ref{lemma:properties-of-seed-link}, e.g., the seed link in
  Figure~\ref{figure:seed-link}.
\item[(C2)] $(K_0,\alpha_0)$, \ldots, $(K_{n-2},\alpha_{n-2})$ are satellite
  configurations satisfying Lemma~\ref{lemma:seed-knot}, e.g., the
  Stevedore configuration in Figure~\ref{figure:stevedore-knot}.
\item[(C3)] $J_0$ is a connected sum of copies of a knot satisfying
  Lemma~\ref{lemma:cochran-teichner-knot}.
\end{enumerate}

We define a two-component $L$ by an iterated satellite construction as
follows: let $J_k=K_{k-1}(\alpha_{k-1},J_{k-1})$ for $k=1,\ldots,n-1$
inductively.  Define $L=L_0(\eta,J_{n-1})$.

The link $L$ can be described alternatively, by reversing the order of
the satellite constructions: define $L_1=L_0(\eta,K_{n-2})$ and
$L_{k}=L_{k-1}(\alpha_{n-k},K_{n-k-1})$ for $k=2,\ldots,n-1$.  Note
that as in Proposition~\ref{proposition:composition-satellite-conf},
$\alpha_{n-k}$ can be viewed as a curve in $E_{K_{n-k}}\subset
E_{L_{k-1}}$ so that the inductive definition makes sense.  Finally
let $L_{n} = L_{n-1}(\alpha_{0},J)$.  Then the link $L_n$ is exactly
our $L$ defined above.

We note that since the curves $\eta$ and $\alpha_k$ are in the
commutator subgroup of $\pi_1(E_{L_0})$ and $\pi_1(E_{K_k})$
respectively, an induction shows that the top level curve $\alpha_0$
lies in the $n$th derived subgroup $\pi_1(E_{L_{n-1}})^{(n)}$.

Observe that if $(L_0,\eta)$ is the one given in
Figure~\ref{figure:seed-link}, then each component of $L_n$ is
unknotted since the union of $\eta$ and any one of the two components
of $L_0$ is a trivial link.

\begin{proposition}
  \label{proposition:height-(n+2)-grope-concordant-to-H}
  The link $L$ is height $n+2$ grope concordant to the Hopf link.
\end{proposition}

\begin{proof}
  Note that $(L_{n-1},\alpha_{0})$ is a satellite configuration of
  height~$n$ by
  Proposition~\ref{proposition:composition-satellite-conf} applied
  inductively.  Therefore it follows that our $L=L_n$ is height $n+2$
  grope concordant to~$L_{n-1}$ by
  Lemma~\ref{lemma:cochran-teichner-knot}.  Recall that for a link $R$
  and a curve $\beta\subset S^3-R$, $R(\beta,J)$ is concordant to
  $R(\beta,J')$ if $J$ is concordant to~$J'$; a concordance is
  obtained by filling in $(S^3-\nu(\beta),R)\times [0,1]$ with the
  exterior of a concordance between $J$ and~$J'$.  Applying this
  inductively to our second description of the links $L_i$, it follows
  that $L_{n-1}$ is concordant to $L_0$ since each $K_i$ is slice.
  Consequently $L_{n-1}$ is concordant to the Hopf link by
  Lemma~\ref{lemma:properties-of-seed-link}.
\end{proof}

\subsection{Proof of the nonexistence of Whitney tower concordance}

The remaining part of this section is devoted to the proof of the
following.  From now on $\Lh$ denotes the Hopf link, and $L$ denotes
our link constructed above.

\begin{theorem}
  \label{theorem:nonexistence-n.5-solvably-cobordism}
  The exterior $E_L$ of $L$ is not $n.5$-solvably cobordant to the
  Hopf link exterior $E_{\Lh}$.
\end{theorem}

By Corollary~\ref{corollary:link-whitney-tower-solvability}, it
follows that our $L$ is not height $n+2.5$ grope (nor Whitney tower)
concordant to the Hopf link.  This completes the proof of
Theorem~\ref{theorem:whitney-tower-concordance-to-hopf-link}.

In the proof of
Theorem~\ref{theorem:nonexistence-n.5-solvably-cobordism}, we combine
our Amenable Signature
Theorem~\ref{theorem:amenable-signature-for-solvable-cobordism} with a
construction of 4-dimensional cobordisms and a higher order
Blanchfield pairing argument, in the same way as done in
\cite[Section~4.3]{Cha:2010-01}.  This is modeled on (but technically
simpler than) an earlier argument due to Cochran, Harvey, and Leidy,
which appeared in~\cite{Cochran-Harvey-Leidy:2009-1}.

\subsubsection*{Cobordism associated to an iterated satellite construction}

Suppose $W$ is an $n.5$-solvable cobordism between $E_L$
and~$E_{\Lh}$.  We recall that associated to a satellite construction
applied to a framed circle $\eta$ in a 3-manifold $Y$ using a knot
$J$, there is a standard cobordism from $M(J)\cup Y$ to $Y(\eta,J)$,
namely $(M(J)\times[0,1] \cup Y\times [0,1])/\sim$, where the tubular
neighborhood of $\eta\subset Y\times 0$ is identified with the solid
tori $(M(J) - \inte E_L)\times 0$ \cite[p.\
1429]{Cochran-Harvey-Leidy:2009-1}.  In particular our satellite
construction gives a standard cobordism $E_k$ from $M(J_k)\cup M(K_k)$
to $M(J_{k+1})$ for $k=0,\ldots,n-2$, and $E_{n-1}$ from
$M(J_{n-1})\cup (E_{L_0}\cup_\partial E_{\Lh})$ to $E_L\cup_\partial
E_{\Lh}$.  Define $W_n=W$, and for $k=n-1,n-2,\ldots,0$, define $W_k$
by
\begin{align*}
  W_k & = E_{k} \amalgover{M(J_{k+1})} E_{k+1} \amalgover{M(J_{k+2})}
  \cdots \amalgover{M(J_{n-1})} E_{n-1} \amalgover{E_L\cupover{\partial}
    E_{\Lh}} W_{n}
  \\
  & = E_{k} \amalgover{M(J_{k+1})} W_{k+1}.
\end{align*}
Note that $\partial W_n=E_{L}\cup_\partial E_{\Lh}$ and $\partial
W_k=M(J_k) \cup M(K_k) \cup M(K_{k+1})\cup \cdots \cup M(K_{n-2})\cup
(E_{L_0}\cup_\partial E_{\Lh})$ for $k<n$.

\subsubsection*{Representations on mixed-type commutator quotients}

To construct solvable representations to which we apply Amenable
Signature
Theorem~\ref{theorem:amenable-signature-for-solvable-cobordism}, we
use mixed-coefficient commutator series $\{\cP^k \pi\}$ as in
\cite[Section~4.1]{Cha:2010-01} and
\cite[Section~3.1]{Cha-Orr:2011-01}.  For the reader's convenience we
repeat the definition: for a group $\pi$ and sequence $\cP=(R_0,
\ldots, R_n)$ where each $R_k$ a commutative ring with unity,
$\cP^k\pi$ is defined inductively by $\cP^0\pi := \pi$ and
\[
\cP^{k+1}\pi := \Ker \Big\{ \cP^k \pi \to \frac{\cP^k\pi}{[\cP^k\pi,
  \cP^k\pi]} \to \frac{\cP^k\pi}{[\cP^k\pi, \cP^k\pi]}\otimesover{\Z}
R_k \big\}
\]
Here we state the following facts, which are easily verified from the
definition. Suppose $\cP=(R_0, \ldots, R_n)$ where $R_k=\Q$ for $k<n$
and $R_n=\Q$ or $\Z/{p}$ for $p$ a fixed prime.  Then for $k\le n$,
$\cP^{k}\pi$ is the $k$th rational derived subgroup.  In particular
$\pi/\cP^{k}\pi$ is PTFA for $k\le n$.  We have $\pi^{(k)} \subset
\cP^k \pi$, and consequently for $G=\pi/\cP^{n+1} \pi$,
$G^{(n+1)}=\{e\}$.  Also, by \cite[Lemma~4.3]{Cha:2010-01}, $G$ is
amenable and in~$D(R_n)$.

For $W_0$ defined above, we have the following:

\newtheorem*{nontriviality-theorem}
{A special case of Theorem~\ref{theorem:meridian-under-mixed-coefficient-quotient}}

\begin{nontriviality-theorem}
  The projection
  \[
  \phi_0\colon \pi_1(W_0) \to G:=\pi_1(W_0)/\cP^{n+1}\pi_1(W_0)
  \]
  sends the meridian of $J_0$ that lies in $M(J_0)\subset \partial
  W_0$ to an element in the abelian subgroup
  $\cP^n\pi_1(W_0)/\cP^{n+1}\pi_1(W_0)$ which has order $\infty$ if
  $R_n=\Q$, and has order $p$ if $R_n=\Z/p$.
\end{nontriviality-theorem}

Its proof is deferred to Section~\ref{subsection:Blanchfield-bordism}.

\subsubsection*{Application of Amenable Signature Theorem}

As an abuse of notation, we denote by $\phi_0$ various homomorphisms
induced by~$\phi_0$ (e.g., the restrictions of $\phi_0$ on
$\pi_1(M(J_0))$, $\pi_1(E_{L_0}\cup_\partial E_{H})$, and
$\pi_1(M(K_k))$).  For a 4-manifold $X$ over $G$, we denote the
$L^2$-signature defect by $S_G(X):=\lsign_G(X)-\sign(X)$.  Then, the
$\rhot$-invariant of $\partial W_0$ is given by
\begin{multline*}
  \rhot(M(J_0),\phi_0)+ \rhot(E_{L_0}\cup_\partial E_H,\phi_0) +
  \sum_{k=0}^{n-2} \rhot(M(K_k),\phi_0)\\
  = \rhot(\partial W,\phi_0) = S_G(W_n) + \sum_{k=0}^{n-1} S_G(E_k).
\end{multline*}
Recall that $W_n$ is assumed to be an $n.5$-solvable cobordism between
$E_H$ and~$E_L$.  Since the abelianization $\pi_1(W_n) \to \Z^2$
decomposes as $\pi_1(W_1) \xrightarrow{\phi_0} G \to G/G^{(1)} =
\Z^2$, we have $\bt_1(E_H;\N G)=0$ by
Theorem~\ref{theorem:two-component-links-with-b^2_1=0}.  Therefore, by
applying our Amenable Signature
Theorem~\ref{theorem:amenable-signature-for-solvable-cobordism} (I),
we obtain
\[
\rhot(E_L\cup E_H,\phi_0)=S_G(W_n)=0.
\]
Also, according to \cite[Lemma~2.4]{Cochran-Harvey-Leidy:2009-1},
$S_G(E_k)=0$ for each~$k$.  It follows that
\[
  \rhot(M(J_0),\phi_0)=-\rhot(E_{L_0}\cup_\partial E_H,\phi_0) -
  \sum_{k=0}^{n-2} \rhot(M(K_k),\phi_0).
\]

Due to Cheeger and Gromov \cite[p.~23]{Cheeger-Gromov:1985-1} (see
also the discussion of \cite[Theorem~2.9]{Cochran-Teichner:2003-1}),
for any closed 3-manifold $M$ there is a bound $C_M$ such that
$|\rhot(M,\phi)| \le C_M$ for any homomorphism $\phi$ of~$\pi_1(M)$.
Therefore, if
\[
\big|\rhot(M(J_0),\phi_0)\big|>C_{E_{L_0}\cup_\partial E_H} +
\sum_{k=0}^{n-2} C_{M(K_k)}
\]
then we derive a contradiction.  That is, $E_L$ and $E_H$
are not $n.5$-solvably cobordant.

The invariant $\rhot(M(J_0),\phi_0)$ can be given explicitly as
follows.  By
Theorem~\ref{theorem:meridian-under-mixed-coefficient-quotient}, the
map $\phi_0$ restricted on $\pi_1(M(J_0))$ has image $\Z$ if $R_n=\Q$,
$\Z/p$ if $R_n=\Z/p$.  Therefore by the $L^2$-induction property and
known computation of the abelian $\rhot$-invariant of a knot, (e.g.,
see \cite[Proposition~5.1]{Cochran-Orr-Teichner:2002-1},
\cite[Corollary~4.3]{Friedl:2003-5},
\cite[Lemma~8.7]{Cha-Orr:2009-01}), we have
\[
\rhot(M(J_0),\alpha) =
\begin{cases}
  \int_{S^1} \sigma_{J_0}(w)\,dw & \text{if }R_n=\Q \\[1ex]
  \sum_{k=0}^{p-1} \sigma_{J_0}(e^{2\pi k\sqrt{-1}/p}) &\text{if }
  R_n=\Z/p
\end{cases}
\]
where $\sigma_{J_0}(\omega)$ is the Levine-Tristram signature function
of~$J_0$.  Combining all these, we have proven the following:

\begin{theorem}
  \label{theorem:not-being-(n.5)-solvable-cobordant}
  Suppose $(L_0,\eta)$, $(K_k,\alpha_k)$ are fixed and satisfy
  {\rm(C1)}, {\rm(C2)} in the beginning of
  Section~\ref{subsection:construction-of-examples}.  If either
  $\int\sigma_{J_0}(w)\,dw$ or $\sum_{k=0}^{p-1} \sigma_{J_0}(e^{2\pi
    k\sqrt{-1}/p})$, for some prime $p$, is sufficiently large, then
  for our link $L$, the exterior $E_L$ is not $n.5$-solvably cobordant
  to~$E_H$.  Consequently $L$ is not height $n+2.5$ Whitney tower
  concordant and not height $n+2.5$ grope concordant to the Hopf link.
\end{theorem}

The last conclusion follows from the first conclusion by applying
Corollary~\ref{corollary:link-whitney-tower-solvability}.

In particular, since the Cochran-Teichner knot $J$ satisfies
Lemma~\ref{lemma:cochran-teichner-knot}, if we take as $J_0$ the
connected sum of sufficiently many copies of $J$, then by
Theorem~\ref{theorem:not-being-(n.5)-solvable-cobordant} our $L$ is
not height $n+2.5$ Whitney tower concordant (and so not height $n+2.5$
grope concordant) to the Hopf link.  On the other hand, by
Proposition~\ref{proposition:composition-satellite-conf} and
Lemma~\ref{lemma:cochran-teichner-knot}~(2), $L$ is height $n+2$ grope
concordant (and so height $n+2$ Whitney tower concordant) to the Hopf
link.  This proves
Theorem~\ref{theorem:whitney-tower-concordance-to-hopf-link}, modulo
the proof of
Theorem~\ref{theorem:meridian-under-mixed-coefficient-quotient}, which
is given next.

\subsection{Blanchfield bordism and nontriviality of solvable
  representations}
\label{subsection:Blanchfield-bordism}

We will complete the proof of
Theorem~\ref{theorem:whitney-tower-concordance-to-hopf-link} by
proving the following nontriviality result:

\begin{theorem}[cf.\
  {\cite[Proposition~8.2]{Cochran-Harvey-Leidy:2009-1}},
  {\cite[Theorem~4.14]{Cha:2010-01}}]
  \label{theorem:meridian-under-mixed-coefficient-quotient}
  For each $k=0,1,\ldots,n-1$, the projection $ \phi_k\colon
  \pi_1(W_k) \to \pi_1(W_k)/\cP^{n-k+1}\pi_1(W_k) $ sends a meridian
  $\mu_k\subset M(J_k)\subset \partial W$ of $J_k$ into the abelian
  subgroup $\cP^{n-k}\pi_1(W_k)/\cP^{n-k+1}\pi_1(W_k)$.  Furthermore,
  $\phi_k(\mu_k)$ has order $p$ if $R_n=\Z/p$ and $k=0$, and has order
  $\infty$ otherwise.
\end{theorem}

We remark that although we use only the case of $k=0$ in the proof of
Theorem~\ref{theorem:whitney-tower-concordance-to-hopf-link}, we state
it in this generality since we need an induction argument for
$k=n-1,\ldots, 0$ in its proof.

Our proof of
Theorem~\ref{theorem:meridian-under-mixed-coefficient-quotient} is a
variation of the higher order Blanchfield pairing technique which was
first introduced by Cochran, Harvey, and Leidy, but different from
arguments in earlier papers (e.g., \cite{Cochran-Harvey-Leidy:2009-1},
\cite{Cochran-Harvey-Leidy:2008-1}, \cite{Cha:2010-01}) as discussed
below.

The notion of certain 4-manifolds called \emph{$n$-bordisms}
\cite[Definition~5.1]{Cochran-Harvey-Leidy:2009-1} plays an important
r\^ole in understanding the behavior of solvable coefficient systems in
earlier works.  Its key property is that if a certain rank condition
is satisfied (see, e.g., \cite[Theorem~5.9,
Lemma~5.10]{Cochran-Harvey-Leidy:2009-1}), an $n$-bordism gives a
submodule that annihilates itself under the higher order Blanchfield
pairing of the boundary, generalizing the fact that the classical
Blanchfield pairing of a slice knot is Witt trivial.  This is an
essential ingredient used in several papers to investigate
higher-order coefficient systems.  For example see
\cite{Cochran-Harvey-Leidy:2009-1, Cochran-Harvey-Leidy:2008-1,
  Horn:2010-1, Cha:2010-01}.

Generalizing this, we consider a 4-dimensional bordism that we call a
\emph{Blanchfield bordism}.  Indeed for our purpose we need to use
Blanchfield bordisms to which prior results of Cochran-Harvey-Leidy
\cite[Theorem~5.9, Lemma~5.10]{Cochran-Harvey-Leidy:2009-1} for
$n$-bordisms do not apply.

% Secondly we use Blanchfield pairings of two-component links, as well
% as Blanchfield pairings of knots used in the iterated satellite
% construction.  A key fact we use is the nondegeneracy of the
% unlocalized Blanchfield pairing of a two-component link with nonzero
% linking number, which is due to Levine~\cite{Levine:1982-1}.

\subsubsection*{Blanchfield bordism}

% Recall that for an Ore domain $\cR$, there exists the skew-field of
% quotients $\K=(\cR-\{0\})^{-1} \cR$.  For an $\cR$-module $M$,
% denote the $\cR$-torsion submodule of $M$ by~$tM$.  , and the rank
% of $M$ by $\rank_{\cR}=\dim_{\K} \K\otimes_{\cR} M$.  We use this
% notation when $\cR$ is a subring of a (commutative) field too.

Throughout this section, $R=\Z/p$ or a subring of $\Q$, and $G$ is
assumed to be a group whose group ring $R G$ is an Ore domain.  Our
standard example to keep in mind is the case of a PTFA group~$G$ (see
\cite[Proposition~2.5]{Cochran-Orr-Teichner:1999-1} and
\cite[Lemma~5.2]{Cha:2010-01}).  We denote the skew-field of quotient
of $R G$ by~$\K G$.  For a module $M$ over an Ore domain, denote the
torsion submodule of $M$ by~$tM$.

\begin{definition}
  \label{definition:blanchfield-bordism}
  Suppose $W$ is a 4-manifold with boundary and $\phi\colon
  \pi_1(W)\to G$ is a homomorphism.  We call $(W,\phi)$ a
  \emph{Blanchfield bordism} if the following is exact.
  \[
  tH_2(W,\partial W;R G) \to tH_1(\partial W;R G) \to tH_1(W;R G)
  \]
\end{definition}

When the choice of $R$ is not clearly understood, we call $(W,\phi)$
an $R$-coefficient Blanchfield bordism.

The key property of a Blanchfield bordism is the following.  For a
3-manifold $M$ endowed with $\phi\colon \pi_1(M)\to G$ and a subring
$\cR$ of $\K G$ containing $R G$, there is the Blanchfield pairing
\[
\Bl_M\colon tH_1(M;\cR) \times tH_1(M;\cR) \to \K G/\cR
\]
defined as in \cite[Theorem~2.13]{Cochran-Orr-Teichner:1999-1}, Then,
the following is shown by known arguments (see, e.g., \cite[Proof of
Theorem~2.3]{Hillman:2002-1}, \cite[Proof of
Theorem~4.4]{Cochran-Orr-Teichner:1999-1}).  We omit details of its
proof.

\begin{theorem}
  \label{theorem:blanchfield-self-annihilating}
  If $(W,\phi\colon \pi_1(W)\to G)$ is a Blanchfield bordism and $R G
  \subset \cR \subset \K G$, then for any 3-manifold
  $M\subset \partial W$,
  \[
  P:=\Ker\{tH_1(M;\cR) \to tH_1(W;\cR)\}
  \]
  satisfies $\Bl_M(P,P)=0$, namely $P$ annihilates $P$ itself.
\end{theorem}

As an example, if $W$ is an $n$-bordism in the sense of
\cite[Definition~5.1]{Cochran-Harvey-Leidy:2009-1}, then for
$\phi\colon \pi_1(W)\to G$ satisfying $G^{(n)}=\{e\}$ and $\dim_{\K G}
H_1(M;\K G)=b_1(M)-1$ for each component $M$ of $\partial W$,
$(W,\phi)$ is an integral Blanchfield bordism by
\cite[Lemma~5.10]{Cochran-Harvey-Leidy:2009-1}.

The following observation provides new examples of Blanchfield
bordisms.

\begin{theorem}
  \label{theorem:betti-number-and-blanchfield-bordism}
  Suppose $W$ is a 4-manifold with boundary, $\phi\colon \pi_1(W)\to
  G$, and there are $\ell_1,\ldots,\ell_m,d_1,\ldots,d_m$ in $H_2(W;R
  G)$ satisfying $\lambda_G(\ell_i,\ell_j)=0$ and $\lambda_G(\ell_i,
  d_j)=\delta_{ij}$ where $\lambda_G$ is the $\Z G$-valued
  intersection pairing on $H_2(W; R G)$.  If $\rank_{R} H_2(W,M;R) \le
  2m$ for some $M\subset \partial W$, then $(W,\phi)$ is a Blanchfield
  bordism.
\end{theorem}

An immediate consequence of
Theorem~\ref{theorem:betti-number-and-blanchfield-bordism} is the
following:

\begin{corollary}
  \label{corollary:solvable-cobordism-is-a-blanchfield-bordism}
  If $W$ is an $n$-solvable cobordism between bordered 3-manifolds $M$
  and $M'$, then for any $\phi\colon\pi_1(W)\to G$ with
  $G^{(n)}=\{e\}$, $(W,\phi)$ is an $R$-coefficient Blanchfield
  bordism for $R=\Z/p$ or any subring of $\Q$.
  % $\dim_{\K G}
  % \Coker\{H_2(\partial W;\K G) \to H_2(W;\K G)\} = b_2(W,M)$.
\end{corollary}

\begin{proof}
  [Proof of
  Theorem~\ref{theorem:betti-number-and-blanchfield-bordism}] 

  We claim:
  \begin{align*}
    2m \ge \rank_{R} H_2(W,M;R) &
    \ge \dim_{\K G} H_2(W,M;\K G) \\
    & \ge \dim_{\K G} \Coker\{H_2(\partial W;\K G) \rightarrow
    H_2(W;\K G)\}.
  \end{align*}
  By applying \cite[Corollary~2.8]{Cochran-Harvey:2006-01} (or
  its $\Z/p$-analogue if $R=\Z/p$) to the chain complex $C_*(W,M;RG)$,
  we obtain the second inequality.  Next, $H_2(W,M;\K G)$ has the
  submodule $\Coker\{H_2(M;\K G) \rightarrow H_2(W;\K G)\}$ which
  surjects onto $\Coker\{H_2(\partial W;\K G) \rightarrow H_2(W;\K
  G)\}$.  From this the third inequality follows. 

  Now the proof is completed by the following fact stated as
  Lemma~\ref{lemma:betti-number-rank} below, which is proven by known
  arguments due to Cochran-Orr-Teichner (see the proof of
  \cite[Lemma~4.5]{Cochran-Orr-Teichner:1999-1}; see also
  \cite[Lemma~5.10]{Cochran-Harvey-Leidy:2009-1}).
\end{proof}

\begin{lemma}
  [\cite{Cochran-Orr-Teichner:1999-1,Cochran-Harvey-Leidy:2009-1}]
  \label{lemma:betti-number-rank}
  If there are $\ell_1,\ldots,\ell_m,d_1,\ldots,d_m$ as in
  Theorem~\ref{theorem:betti-number-and-blanchfield-bordism} and the
  cokernel of $H_2(\partial W;\K G) \rightarrow H_2(W;\K G)$ has $\K
  G$-dimension $\le 2m$, then $(W,\phi)$ is a Blanchfield bordism (and
  the equality holds).
\end{lemma}

\begin{remark}
  \label{remark:dimension-of-cokernel}
  From the above proof, we also see that in
  Corollary~\ref{corollary:solvable-cobordism-is-a-blanchfield-bordism}
  the cokernel of $H_2(\partial W;\K G) \to H_2(W;\K G)$ has the
  ``right'' dimension, namely $b_2(W,M)$.
\end{remark}

\begin{proof}
[Proof of
  Theorem~\ref{theorem:meridian-under-mixed-coefficient-quotient}]

Recall the conclusion of
Theorem~\ref{theorem:meridian-under-mixed-coefficient-quotient}: the
projection
\[
\phi_k\colon \pi_1(W_k) \to \pi_1(W_k)/\cP^{n-k+1}\pi_1(W_k)
\]
sends the meridian $\mu_k\subset M(J_k)\subset \partial W$ to an
element in $\cP^{n-k}\pi_1(W_k)/\cP^{n-k+1}\pi_1(W_k)$, which has
order $\infty$ if $k\ne 0$ or $R_n=\Q$, $p$ otherwise.  In fact it
suffices to show the nontriviality of $\phi_k(\mu_k)\in
\cP^{n-k}\pi_1(W_k)/\cP^{n-k+1}\pi_1(W_k)$, since
$\cP^{n-k}\pi_1(W_k)/\cP^{n-k+1}\pi_1(W_k)$ is a torsion free abelian
group (if $k\ne 0$ or $R_n=\Q$) or a vector space over~$\Z/p$
(otherwise).

We use an induction on $k=n-1,n-2,\ldots,0$.  For the case $k=n-1$, we
start by considering $\phi_n\colon \pi_1(W_n)\to
G:=\pi_1(W_n)/\cP^{1}\pi_1(W_n)$.  Recall $G\cong
H_1(W_n)/\text{torsion}=\Z^2$.  Let $R=R_1$ and $A=H_1(E_L;R G)\cong
H_1(E_{L_0};R[x^{\pm1},y^{\pm1}])$, the Alexander module.

%   We observe
% the following:
% \begin{enumerate}
% \item[(i)] $A$ is torsion and the Blanchfield pairing $\Bl_L$ on $A$
%   is nondegenerated.  This is a general fact for linking number one
%   two-component links by Levine~\cite{Levine:1982-1} (or follows from
%   a direct computation of~$\Bl_L$ in our case).
% \item[(ii)] $[\eta]\in A$ generates~$A$ by
%   Lemma~\ref{lemma:properties-of-seed-link}.
% \item[(iii)] $(W_n,\phi_n)$ is a Blanchfield bordism by
%   Theorem~\ref{corollary:solvable-cobordism-is-a-blanchfield-bordism}.
% \end{enumerate}

We need the fact that ($A$ is torsion and) the Blanchfield pairing
$\Bl_{L}=\Bl_{L_0}$ on $A$ is nondegenerate.  This is a general fact
for linking number one two-component links due to
Levine~\cite{Levine:1982-1}, or can be seen by straightforward
computation in our case.

Recall that we use the curve $\eta\subset E_{L_0}$ in the satellite
construction.  Denote a paralel copy of $\eta$ in $E_L$ by $\eta$ as
an abuse of notation.  $\Bl_{L}(\eta,\eta)$ is nonzero, since $[\eta]$
generates the nontrivial torsion module $A$ by
Lemma~\ref{lemma:properties-of-seed-link} and $\Bl_{L}$ on $A$ is
nondegenerate.  Therefore $[\eta]\notin P=\Ker\{A \to H_1(W_n; R
G)\}$, since $P\subset P^{\perp}$ by Theorem
~\ref{theorem:blanchfield-self-annihilating}.  Since $\cP^2\pi_1(W_n)$
is the kernel of $\cP^1\pi_1(W_n) \rightarrow H_1(W_n;R G)$ by
definition (see \cite[Section 4.1]{Cha:2010-01}), it follows that
$[\eta]\notin \cP^2\pi_1(W_n)$.  As in \cite[Assertion~1 in
Section~5.2]{Cha:2010-01}, we have
\[
\cP^{n-k}\pi_1(W_k)/\cP^{n-k+1}\pi_1(W_k) \cong
\cP^{n-k}\pi_1(W_{k+1})/\cP^{n-k+1}\pi_1(W_{k+1}).
\]
Looking at the $k=n-1$ case and observing that $\eta$ is isotopic to
$\mu_{n-1}\subset M(J_{n-1})$ in $W_{n-1}$, it follows that
$[\mu_{n-1}]$ is nontrivial in
$\cP^1\pi_1(W_{k+1})/\cP^2\pi_1(W_{k+1})$.  This is the desired
conclusion for $k=n-1$.

Now we assume that the conclusion is true for all $i>k$.  Let
$G=\pi_1(W_{k+1})/\cP^{n-k}\pi_1(W_{k+1})$, $R=R_{n-k}$, and
$A=H_1(M(J_{k+1});R G)$ for convenience.

We claim that $(W_{k+1},\phi_{k+1}\colon\pi_1(W_{k+1})\to G)$ is an
$R$-coefficient Blanchfield bordism.
% \begin{enumerate}
% \item[(i)] $A\cong H_1(M(K_{k});R G)$ is torsion and the Blanchfield
%   pairing $\Bl_{K_k}$ on $A$ is nondegenerated.  This follows from
%   \cite[Theorem~4.7]{Leidy:2006-1}, \cite[Theorem~5.16]{Cha:2003-1},
%   \cite[Lemma~6.5, Theorem~6.6]{Cochran-Harvey-Leidy:2009-1} (see also
%   \cite[Theorem~5.4]{Cha:2010-01})
% \item[(ii)] $[\alpha_k]\in A$ generates~$A$ by
%   Lemma~\ref{lemma:seed-knot}.
% \item[(iii)] $(W_{k+1},\phi_{k+1}\colon \pi_1(W_{k+1})\to G)$ is a
%   Blanchfield bordism.
% \end{enumerate}
To show this we need the following lemma:

\begin{lemma}
  Suppose $W$ is a 4-manifold with a boundary component $N'=N(\eta,J)$
  obtained by a satellite construction.  Let $E$ be the associated
  standard cobordism from $M(J)\cup N$ to~$N'$, and let $V=W\cup_{N'}
  E$.  If $\phi\colon \pi_1(V)\to G$ sends $[\eta]$ to a nontrivial
  element, then the inclusion $W\to V$ induces
  \[
  \Coker\{H_2(\partial V;\K G) \rightarrow H_2(V;\K G)\} \cong
  \Coker\{H_2(\partial W;\K G) \rightarrow H_2(W;\K G)\}.
  \]
\end{lemma}

\begin{proof}
  The proof is a straightforward Mayer-Vietoris argument.  We give
  an outline below.  The cobordism $E$ is defined to be
  $M(J)\times[0,1]\cup N\times[0,1]/\sim$, where the solid torus
  $U:=\overline{M(J)-E_J} \cong S^1\times D^2$ in $M(J)=M(J)\times 1$
  is identified with a tubular neighborhood of $\eta\subset N=N\times
  1$.  Applying Mayer-Vietoris to this, one sees that $H_i(N;\K G)
  \cong H_i(E;\K G)$ for $i=1,2$.  Here one needs that $H_1(U;\K G) =
  H_1(M(J);\K G)=0$, which are consequences of $\phi([\eta])\ne e$ by
  \cite[Proposition 2.11]{Cochran-Orr-Teichner:1999-1}.  Similarly one
  sees that $H_i(N';\K G) \cong H_i(E;\K G)$ for $i=1,2$.  This says
  that the cobordism $E$ looks like a cylinder to the eyes of
  $H_i(-;\K G)$ for $i=1,2$.  Applying Mayer-Vietoris to $V=W\cup_{N'}
  E$, the desired conclusion follows.
\end{proof}

Returning to the case of our $W_{k+1}$, the induction hypothesis
$\phi_\ell([\mu_\ell])\ne e$ for $\ell\ge k+1$ enables us to apply the
lemma above repeatedly.  From this we obtain
\[
\Coker\{H_2(\partial W_{k+1};\K G)\rightarrow H_2(W_{k+1};\K G)\}
\cong \Coker\{H_2(\partial W_{n};\K G)\rightarrow H_2(W_{n};\K G)\}.
\]
By
Corollary~\ref{corollary:solvable-cobordism-is-a-blanchfield-bordism},
Remark~\ref{remark:dimension-of-cokernel} and
Lemma~\ref{lemma:betti-number-rank}, it follows that
$(W_{k+1},\phi_{k+1})$ is a Blanchfield bordism.

Now we proceed similarly to the $k=n-1$ case.  We need the following
fact which is due to \cite[Theorem~4.7]{Leidy:2006-1},
\cite[Theorem~5.16]{Cha:2003-1}, \cite[Lemma~6.5,
Theorem~6.6]{Cochran-Harvey-Leidy:2009-1} (see also
\cite[Theorem~5.4]{Cha:2010-01} for a summarized version):
\[
A\cong H_1(M(K_{k});R G) \cong R G \otimes_{R[t^{\pm1}]}
H_1(M(K_k);R[t^{\pm1}])
\]
and the classical Blanchfield pairing on $H_1(M(K_k);R[t^{\pm1}])$
vanishes at $(x,y)$ if the Blanchfield pairing $\Bl$ on $A$ vanishes
at $(1\otimes x, 1\otimes y)$.  Using this, one sees that
$\Bl(1\otimes[\alpha_k],1\otimes[\alpha_k])\ne 0$, since $[\alpha_k]$
generates the nontrivial module $H_1(M(K_{k});R G)$ by
Lemma~\ref{lemma:seed-knot} and the classical Blanchfield pairing of a
knot is nonsingular.  Therefore $[\alpha_k] \notin P=\Ker\{A\to
H_1(W_{k+1};RG)\}$ by
Theorem~\ref{theorem:blanchfield-self-annihilating} applied to the
Blanchfield bordism $(W_{k+1},\phi_{k+1})$.  Finally, proceeding in
the exactly same way as the last part of the $k=n-1$ case, we conclude
that $[\mu_k]$ is nontrivial in
$\cP^{n-k}\pi_1(W_{k})/\cP^{n-k+1}\pi_1(W_{k})$.  This completes the
proof of
Theorem~\ref{theorem:meridian-under-mixed-coefficient-quotient}.
\end{proof}

% \subsubsection*{Acknowledgments} 

% This work was supported by the National Research Foundation of
% Korea(NRF) grant funded by the Korea government(MEST), No.\
% 2010--0011629 and 2010--0029638.

% Bibliography file of Jae Choon Cha <jccha@postech.ac.kr>
\bibliographystyle{amsalpha}
\renewcommand{\MR}[1]{}
\bibliography{research}

\def\cprime{$'$}
\providecommand{\bysame}{\leavevmode\hbox to3em{\hrulefill}\thinspace}
\providecommand{\MR}{\relax\ifhmode\unskip\space\fi MR }
% \MRhref is called by the amsart/book/proc definition of \MR.
\providecommand{\MRhref}[2]{%
  \href{http://www.ams.org/mathscinet-getitem?mr=#1}{#2}
}
\providecommand{\href}[2]{#2}
\begin{thebibliography}{CKRS12}

\bibitem[CF]{Cha-Friedl:2010-01}
Jae~Choon Cha and Stefan Friedl, \emph{Twisted torsion invariants and link
  concordance}, arXiv:1001.0926, to appear in Forum Math.

\bibitem[CG85]{Cheeger-Gromov:1985-1}
Jeff Cheeger and Mikhael Gromov, \emph{Bounds on the von {N}eumann dimension of
  {$L\sp 2$}-cohomology and the {G}auss-{B}onnet theorem for open manifolds},
  J. Differential Geom. \textbf{21} (1985), no.~1, 1--34. \MR{MR806699
  (87d:58136)}

\bibitem[CH05]{Cochran-Harvey:2004-1}
Tim~D. Cochran and Shelly Harvey, \emph{Homology and derived series of groups},
  Geom. Topol. \textbf{9} (2005), 2159--2191 (electronic). \MR{MR2209369
  (2007c:20120)}

\bibitem[CH08]{Cochran-Harvey:2006-01}
\bysame, \emph{Homology and derived series of groups. {II}. {D}wyer's theorem},
  Geom. Topol. \textbf{12} (2008), no.~1, 199--232. \MR{MR2377249}

\bibitem[Cha]{Cha:2010-01}
Jae~Choon Cha, \emph{Amenable ${L}^2$-theoretic methods and knot concordance},
  arXiv:1010.1058.

\bibitem[Cha07]{Cha:2003-1}
\bysame, \emph{The structure of the rational concordance group of knots}, Mem.
  Amer. Math. Soc. \textbf{189} (2007), no.~885, x+95. \MR{MR2343079}

\bibitem[Cha08]{Cha:2006-1}
\bysame, \emph{{Topological minimal genus and $L^2$-signatures}}, Algebr. Geom.
  Topol. \textbf{8} (2008), 885--909.

\bibitem[Cha09]{Cha:2007-2}
\bysame, \emph{Structure of the string link concordance group and
  {H}irzebruch-type invariants}, Indiana Univ. Math. J. \textbf{58} (2009),
  no.~2, 891--927. \MR{MR2514393}

\bibitem[Cha10]{Cha:2007-1}
\bysame, \emph{Link concordance, homology cobordism, and {H}irzebruch-type
  defects from iterated $p$-covers}, J. Eur. Math. Soc. (JEMS) \textbf{12}
  (2010), no.~3, 555--610.

\bibitem[CHL08]{Cochran-Harvey-Leidy:2008-1}
Tim~D. Cochran, Shelly Harvey, and Constance Leidy, \emph{Link concordance and
  generalized doubling operators}, Algebr. Geom. Topol. \textbf{8} (2008),
  no.~3, 1593--1646. \MR{MR2443256 (2009h:57014)}

\bibitem[CHL09]{Cochran-Harvey-Leidy:2009-1}
\bysame, \emph{Knot concordance and higher-order {B}lanchfield duality}, Geom.
  Topol. \textbf{13} (2009), no.~3, 1419--1482. \MR{MR2496049 (2009m:57006)}

\bibitem[CK99a]{Cha-Ko:1999-2}
Jae~Choon Cha and Ki~Hyoung Ko, \emph{On equivariant slice knots}, Proc. Amer.
  Math. Soc. \textbf{127} (1999), no.~7, 2175--2182. \MR{2000a:57006}

\bibitem[CK99b]{Cha-Ko:1999-1}
\bysame, \emph{Signature invariants of links from irregular covers and
  non-abelian covers}, Math. Proc. Cambridge Philos. Soc. \textbf{127} (1999),
  no.~1, 67--81. \MR{2000d:57014}

\bibitem[CK08a]{Cha-Kim:2008-1}
Jae~Choon Cha and Taehee Kim, \emph{Covering link calculus and iterated {B}ing
  doubles}, Geom. Topol. \textbf{12} (2008), no.~4, 2173--2201. \MR{MR2431018
  (2009f:57007)}

\bibitem[CK08b]{Cochran-Kim:2004-1}
Tim~D. Cochran and Taehee Kim, \emph{Higher-order {A}lexander invariants and
  filtrations of the knot concordance group}, Trans. Amer. Math. Soc.
  \textbf{360} (2008), no.~3, 1407--1441 (electronic). \MR{MR2357701
  (2008m:57008)}

\bibitem[CKRS12]{Cha-Kim-Ruberman-Strle:2010-01}
Jae~Choon Cha, Taehee Kim, Daniel Ruberman, and Saso Strle, \emph{Smooth
  concordance of links topologically concordant to the hopf link}, Bull. Lond.
  Math. Soc. \textbf{44} (2012), no.~3, 443--450.

\bibitem[CLR08]{Cha-Livingston-Ruberman:2006-1}
Jae~Choon Cha, Charles Livingston, and Daniel Ruberman, \emph{Algebraic and
  {H}eegaard-{F}loer invariants of knots with slice {B}ing doubles}, Math.
  Proc. Cambridge Philos. Soc. \textbf{144} (2008), no.~2, 403--410.
  \MR{MR2405897}

\bibitem[CO]{Cha-Orr:2011-01}
Jae~Choon Cha and Kent~E. Orr, \emph{Hidden torsion, 3-manifolds, and homology
  cobordism}, arXiv:1101.4092, to appear in J. Topol.

\bibitem[CO93]{Cochran-Orr:1993-1}
Tim~D. Cochran and Kent~E. Orr, \emph{Not all links are concordant to boundary
  links}, Ann. of Math. (2) \textbf{138} (1993), no.~3, 519--554.
  \MR{95c:57042}

\bibitem[CO12]{Cha-Orr:2009-01}
Jae~Choon Cha and Kent~E. Orr, \emph{${L}^2$-signatures, homology localization,
  and amenable groups}, Comm. Pure Appl. Math. \textbf{65} (2012), 790--832.

\bibitem[COT03]{Cochran-Orr-Teichner:1999-1}
Tim~D. Cochran, Kent~E. Orr, and Peter Teichner, \emph{Knot concordance,
  {W}hitney towers and {$L\sp 2$}-signatures}, Ann. of Math. (2) \textbf{157}
  (2003), no.~2, 433--519. \MR{1 973 052}

\bibitem[COT04]{Cochran-Orr-Teichner:2002-1}
\bysame, \emph{Structure in the classical knot concordance group}, Comment.
  Math. Helv. \textbf{79} (2004), no.~1, 105--123. \MR{MR2031301 (2004k:57005)}

\bibitem[CS74]{Cappell-Shaneson:1974-1}
Sylvain~E. Cappell and Julius~L. Shaneson, \emph{The codimension two placement
  problem and homology equivalent manifolds}, Ann. of Math. (2) \textbf{99}
  (1974), 277--348. \MR{49 \#3978}

\bibitem[CST11]{Conant-Teichner-Schneiderman:2011-1}
Jim Conant, Rob Schneiderman, and Peter Teichner, \emph{Higher-order
  intersections in low-dimensional topology}, Proc. Natl. Acad. Sci. USA
  \textbf{108} (2011), no.~20, 8131--8138. \MR{2806650}

\bibitem[CST12]{Conant-Teichner-Schneiderman:2012-1}
\bysame, \emph{{W}hitney tower concordance of classical links}, Geom. Topol.
  \textbf{16} (2012), no.~20, 1419--1479.

\bibitem[CT07]{Cochran-Teichner:2003-1}
Tim~D. Cochran and Peter Teichner, \emph{Knot concordance and von {N}eumann
  {$\rho$}-invariants}, Duke Math. J. \textbf{137} (2007), no.~2, 337--379.
  \MR{MR2309149 (2008f:57005)}

\bibitem[CW03]{Chang-Weinberger:2003-1}
Stanley Chang and Shmuel Weinberger, \emph{On invariants of {H}irzebruch and
  {C}heeger-{G}romov}, Geom. Topol. \textbf{7} (2003), 311--319 (electronic).
  \MR{MR1988288 (2004c:57052)}

\bibitem[Dav06]{Davis:2006-1}
James~F. Davis, \emph{A two component link with {A}lexander polynomial one is
  concordant to the {H}opf link}, Math. Proc. Cambridge Philos. Soc.
  \textbf{140} (2006), no.~2, 265--268. \MR{2212279 (2006k:57010)}

\bibitem[FP]{Friedl-Powell:2011-1}
Stefan Friedl and Mark Powell, \emph{Links not concordant to the hopf link},
  arXiv:1105.2773.

\bibitem[FP12]{Friedl-Powell:2010-1}
\bysame, \emph{An injectivity theorem for {C}asson-{G}ordon type
  representations relating to the concordance of knots and links}, Bull. Korean
  Math. Soc. \textbf{49} (2012), no.~2, 395--409.

\bibitem[FQ90]{Freedman-Quinn:1990-1}
Michael~H. Freedman and Frank Quinn, \emph{Topology of 4-manifolds}, Princeton
  Mathematical Series, vol.~39, Princeton University Press, Princeton, NJ,
  1990. \MR{MR1201584 (94b:57021)}

\bibitem[Fre84]{Freedman:1984-1}
Michael~H. Freedman, \emph{The disk theorem for four-dimensional manifolds},
  Proceedings of the {I}nternational {C}ongress of {M}athematicians, {V}ol.\ 1,
  2 ({W}arsaw, 1983) (Warsaw), PWN, 1984, pp.~647--663. \MR{MR804721
  (86m:57016)}

\bibitem[Fri05a]{Friedl:2003-1}
Stefan Friedl, \emph{Link concordance, boundary link concordance and
  eta-invariants}, Math. Proc. Cambridge Philos. Soc. \textbf{138} (2005),
  no.~3, 437--460. \MR{MR2138572}

\bibitem[Fri05b]{Friedl:2003-5}
\bysame, \emph{{$L\sp 2$}-eta-invariants and their approximation by unitary
  eta-invariants}, Math. Proc. Cambridge Philos. Soc. \textbf{138} (2005),
  no.~2, 327--338. \MR{MR2132174 (2006a:57014)}

\bibitem[Har08]{Harvey:2006-1}
Shelly Harvey, \emph{Homology cobordism invariants and the
  {C}ochran-{O}rr-{T}eichner filtration of the link concordance group}, Geom.
  Topol. \textbf{12} (2008), 387--430.

\bibitem[Hil02]{Hillman:2002-1}
Jonathan Hillman, \emph{Algebraic invariants of links}, Series on Knots and
  Everything, vol.~32, World Scientific Publishing Co. Inc., River Edge, NJ,
  2002. \MR{2003k:57014}

\bibitem[Hor10]{Horn:2010-1}
Peter~D. Horn, \emph{The non-triviality of the grope filtrations of the knot
  and link concordance groups}, Comment. Math. Helv. \textbf{85} (2010), no.~4,
  751--773. \MR{2718138}

\bibitem[Kaw78]{Kawauchi:1978-1}
Akio Kawauchi, \emph{On the {A}lexander polynomials of cobordant links}, Osaka
  J. Math. \textbf{15} (1978), no.~1, 151--159. \MR{58 \#7599}

\bibitem[Lei06]{Leidy:2006-1}
Constance Leidy, \emph{Higher-order linking forms for knots}, Comment. Math.
  Helv. \textbf{81} (2006), no.~4, 755--781. \MR{MR2271220 (2007g:57011)}

\bibitem[Lev82]{Levine:1982-1}
Jerome~P. Levine, \emph{The module of a {$2$}-component link}, Comment. Math.
  Helv. \textbf{57} (1982), no.~3, 377--399. \MR{84h:57003}

\bibitem[Lev94]{Levine:1994-1}
\bysame, \emph{Link invariants via the eta invariant}, Comment. Math. Helv.
  \textbf{69} (1994), no.~1, 82--119. \MR{95a:57009}

\bibitem[Lev07]{Levine:2007-1}
\bysame, \emph{Concordance of boundary links}, J. Knot Theory Ramifications
  \textbf{16} (2007), no.~9, 1111--1120. \MR{2375818 (2009c:57042)}

\bibitem[L{\"u}c98]{Lueck:1998-1}
Wolfgang L{\"u}ck, \emph{Dimension theory of arbitrary modules over finite von
  {N}eumann algebras and {$L\sp 2$}-{B}etti numbers. {I}. {F}oundations}, J.
  Reine Angew. Math. \textbf{495} (1998), 135--162. \MR{MR1603853 (99k:58176)}

\bibitem[L{\"u}c02]{Lueck:2002-1}
\bysame, \emph{{$L\sp 2$}-invariants: theory and applications to geometry and
  {$K$}-theory}, Ergebnisse der Mathematik und ihrer Grenzgebiete. 3. Folge. A
  Series of Modern Surveys in Mathematics [Results in Mathematics and Related
  Areas. 3rd Series. A Series of Modern Surveys in Mathematics], vol.~44,
  Springer-Verlag, Berlin, 2002. \MR{MR1926649 (2003m:58033)}

\bibitem[Mur65]{Murasugi:1965-1}
Kunio Murasugi, \emph{On a certain numerical invariant of link types}, Trans.
  Amer. Math. Soc. \textbf{117} (1965), 387--422. \MR{30 #1506}

\bibitem[Sch06]{Schneiderman:2006-1}
Rob Schneiderman, \emph{Whitney towers and gropes in 4-manifolds}, Trans. Amer.
  Math. Soc. \textbf{358} (2006), no.~10, 4251--4278 (electronic). \MR{2231378
  (2007e:57018)}

\bibitem[Sco05]{Scorpan:2005-1}
Alexandru Scorpan, \emph{The wild world of 4-manifolds}, American Mathematical
  Society, Providence, RI, 2005. \MR{2136212 (2006h:57018)}

\bibitem[Str74]{Strebel:1974-1}
Ralph Strebel, \emph{Homological methods applied to the derived series of
  groups}, Comment. Math. Helv. \textbf{49} (1974), 302--332. \MR{MR0354896 (50
  \#7373)}

\bibitem[Tri69]{Tristram:1969-1}
A.~G. Tristram, \emph{Some cobordism invariants for links}, Proc. Cambridge
  Philos. Soc. \textbf{66} (1969), 251--264. \MR{40 #2104}

\end{thebibliography}

\end{document}